\documentclass{amsart}
\usepackage[utf8x]{inputenc}
\usepackage{amsmath, amsthm, amssymb}
\usepackage[usenames,dvipsnames,svgnames,table]{xcolor}
\usepackage[margin=1.3in]{geometry}
\usepackage{enumitem}
\usepackage{dsfont}
\usepackage{cleveref}

\newtheorem{theorem}{Theorem}[section]
\newtheorem{proposition}[theorem]{Proposition}
\newtheorem{lemma}[theorem]{Lemma}
\newtheorem{definition}[theorem]{Definition}
\newtheorem*{remark}{Remark}

\newcommand{\N}{\mathbb N}

\newcommand{\R}{\mathbb R}
\newcommand{\T}{\mathbb T}

\newcommand{\cW}{W}
\newcommand{\eps}{\varepsilon}
\newcommand{\e}{\eps}

\newcommand{\dd}{\, \mathrm{d}}

\newcommand{\tr}{\mbox{tr}}
\newcommand{\Tr}{\mbox{tr}}

\newcommand{\1}{\mathds{1}}
\newcommand{\vv}{\langle v\rangle}
\newcommand{\vvO}{\langle v_0\rangle}

\newcommand{\ul}{\rm ul}
\newcommand{\vve}{\langle v_\e\rangle}

\DeclareMathOperator{\supp}{supp}

\newcommand{\Ckin}{C_{\rm kin}}

\numberwithin{equation}{section}

\title[Solutions of the Landau equation with rough data]{Local solutions of the Landau equation with rough, slowly decaying initial data}
\author{Christopher Henderson}
\address{Department of Mathematics, University of Arizona, Tuscon, AZ 85721}
\email{ckhenderson@math.arizona.edu}
\author{Stanley Snelson}
\address{Department of Mathematical Sciences, Florida Institute of Technology, Melbourne, FL 32901}
\email{ssnelson@fit.edu}
\author{Andrei Tarfulea}
\address{Department of Mathematics, Louisiana State University, Baton Rouge, LA 70803}
\email{tarfulea@lsu.edu}
\thanks{
AT was partially supported by NSF grant DMS-1816643. SS was partially supported by a Ralph E. Powe Award from ORAU.  CH was partially supported by NSF grant DMS-1907853.}

\begin{document}

\begin{abstract}
We consider the Cauchy problem for the spatially inhomogeneous Landau equation with soft potentials in the case of large (i.e. non-perturbative) initial data. We construct a solution for any bounded, measurable initial data with uniform polynomial decay in the velocity variable, and that satisfies a technical lower bound assumption (but can have vacuum regions). For uniqueness in this weak class, we have to make the additional assumption that the initial data is H\"older continuous. Our hypotheses are much weaker, in terms of regularity and decay, than previous large-data well-posedness results in the literature. We also derive a continuation criterion for our solutions that is, for the case of very soft potentials, an improvement over the previous state of the art.
\end{abstract}

\maketitle

\section{Introduction}

We consider the spatially inhomogeneous Landau equation in the whole space: for $(t,x,v) \in \R_+\times \R^3\times\R^3$, the solution $f(t,x,v)\geq 0$ satisfies
\begin{equation}\label{e:landau}
(\partial_t + v\cdot \nabla_x) f = Q_L(f,f) = \nabla_v\cdot (\bar a^f\nabla_v f) + \bar b^f\cdot \nabla_v f + \bar c^f f,
\end{equation}
where the nonlocal coefficients are defined by
\begin{equation}\label{e:coeffs}
\begin{split}
\bar a^f(t,x,v) &:= a_{\gamma}\int_{\R^3} \left( I - \frac w {|w|} \otimes \frac w {|w|}\right) |w|^{\gamma + 2} f(t,x,v-w) \dd w,\\
\bar b^f(t,x,v) &:= b_{\gamma}\int_{\R^3} |w|^\gamma w f(t,x,v-w)\dd w,\\
\bar c^f(t,x,v) &:= c_{\gamma}\int_{\R^3} |w|^\gamma f(t,x,v-w)\dd w, 
\end{split}
\end{equation}
and $a_{\gamma}, b_\gamma, c_\gamma$ are constants with $a_\gamma, c_\gamma > 0$. We are concerned with the case of \emph{soft potentials}, i.e. $\gamma \in [-3,0)$. When $\gamma = -3$, the formula for $\bar c^f$ must be replaced by $\bar c^f = cf$. In fact, the constants are such that $\bar b^f_i = -\sum \partial_{v_j} \bar a^f_{ij}$, so that if $f\in C^2_v$, \eqref{e:landau} may equivalently be written in non-divergence form: 
\begin{equation}\label{e:landau_nd}
	(\partial_t  + v\cdot \nabla_x) f = \tr(\bar a^f D_v^2 f) + \bar c^f f.
\end{equation}
We use both forms of the equation.

The $\gamma = -3$ case of \eqref{e:landau}, called the Landau-Coulomb equation, is used in plasma physics to model the density $f(t,x,v)$ of charged particles in phase space as it evolves in time. For $\gamma > -3$, the equation arises as a formal limit of the Boltzmann equation as grazing collisions (such that the angle between pre- and post-collisional velocites goes to zero) predominate. In addition to its own interest, part of the motivation for studying the Landau equation (particularly the non-Coulomb cases) is to shed light on the Boltzmann equation without angular cutoff 
(see, e.g., \cite{lifschitzpitaevskii,chapmancowling} or \cite{alexandre2004landau} for the physical background). 

Until relatively recently, the theory of classical solutions of the inhomogeneous Landau equation focused mainly on 
global-in-time solutions that are perturbations of a (global) Maxwellian equilibrium state (i.e. $\mu(t,x,v) = c_1 e^{-c_2|v|^2}$ with $c_1,c_2>0$). 
%
This study began with the work of Guo \cite{guo2002periodic} in 2002 and was subsequently extended by many authors; 
see Subsection \ref{s:work} below for a partial bibliography.

In the last few years, there has been a lot of progress on the non-perturbative case, including conditional regularity estimates 
and short-time existence results for large (i.e. non-perturbative) initial data. 
A major motivation for these works is the outstanding open problem of global existence for \eqref{e:landau} with large initial data.

Our goal in this article is to improve the local well-posedness theory for \eqref{e:landau}. There are two previous large-data existence results we are aware of: He-Yang \cite{he2014boltzmannlandau} constructed solutions for the case $\gamma = -3$ with spatially periodic initial data satisfying $ |v|^P f_{\rm in} \in H^7(\R^6)$, for $P$ an exponent on the order of $100$, and in our previous work \cite{HST2018landau}, we constructed solutions for initial data satisfying $e^{\rho |v|^2} f_{\rm in} \in H^4(\R^6)$ for some $\rho >0$.\footnote{More precisely, \cite{HST2018landau} took $f_{\rm in}$ in a ``uniformly local'' space $H^4_{\rm ul}$, which imposes no decay as $|x|\to \infty$.} The strength of these hypotheses, in terms of regularity and decay, is comparable to the state of the art for the non-cutoff Boltzmann equation \cite{amuxy2010regularizing, amuxy2011bounded, amuxy2013mild} but seems unsatisfactory when compared to (i) what is needed to make sense of the equation and (ii) the conditions required for good \emph{a priori} control of solutions, which are zeroth-order and impose only mild polynomical decay in $v$ (see \Cref{s:conditional}). 
The results in this paper go a long way toward bridging this gap, by proving existence of solutions in a merely polynomially-weighted $L^\infty$ space, which requires no control on derivatives or exponential moments of $f_{\rm in}$. To prove uniqueness, we need the initial data to additionally be H\"older continuous. We also improve the conditions under which solutions can be extended past a given time.

\subsection{Main results}
In order to state our results, we define weighted Lebesgue spaces as follows: let $\vv = \sqrt{1+|v|^2}$, $1\leq p \leq \infty$, $k\in\R$, and $\Omega$ be any set of the form $\R^3$ or $\omega \times \R^3$ with the $\R^3$ variable denoted $v$, then define
\[
	\|g\|_{L^{p,k}(\Omega)} = \| \vv^k g \|_{L^p(\Omega)}.
\]
We also require the kinetic H\"older spaces, denoted $\Ckin^\alpha$ and $\Ckin^{2,\alpha}$ for $\alpha\in(0,1)$ and defined in \Cref{sec:notation}, that are analogous to the standard parabolic H\"older spaces.  We use the standard H\"older spaces, $C^\alpha$ and $C^{2,\alpha}$, as well. 

One quantity that plays a large role in our analysis is
\begin{equation}\label{e:Psi}
	\Psi(t)
		:= \begin{cases} \| f\|_{L^\infty_{t,x} L^{1,2}_v([0,t]\times\R^6)}, &\gamma \in (-2,0),\\
			\|f\|_{L^\infty_{t,x} L^{1}_v ([0,t]\times\R^6) }+ \|f\|_{L^\infty_{t,x}L^p_v(\R^6)}, &\gamma\in [-3,-2],
		\end{cases}
\end{equation}
 where $p>3/(3+\gamma)$ and $p=\infty$ for $\gamma = -3$. In the case $\gamma\in[0,2]$, $\Psi$ is related to the hydrodynamic quantities of mass density and energy density (see Section \ref{s:conditional} below).  Notice that $\Psi$ is bounded by the $L^{\infty,k}$ norm of $f$ if $k > 5$.

We require the following condition on our initial data, recalled from \cite{HST2018landau}. It ensures that the self-generating lower bounds on our solution are uniform in $x$ (after a small time has passed):
\begin{definition}\label{d:well}
A function $g: \R^3 \times \R^3 \to [0,\infty)$ is \emph{well-distributed with parameters $R, \delta, r>0$} if, for every $x \in \R^3$, there exists $x_m \in B_{R}(x)$ and $v_m \in B_R(0)$ such that $g \geq \delta \1_{B_r(x_m,v_m)}$.
\end{definition}

Our first result establishes the existence of solutions in $L^{\infty,k}([0,T]\times \R^6)$:

\begin{theorem}\label{t:exist}
	Let $k> \max\{5,15/(5+\gamma)\}$. Assume $\|f_{\rm in}\|_{L^{\infty,k}(\R^6)} \leq K$, $f_{\rm in}(x,v) \geq 0$ in $\R^6$, and $f_{\rm in}$ is well-distributed with parameters $\delta$, $r$, and $R$. Then:
	\begin{enumerate}[label=(\roman*)]
	\item {\bf (Existence)}  There exists $T>0$ and $f \geq 0$, such that, for any compact $\Omega \subset (0,T]\times \R^6$, $f \in \Ckin^{2,\alpha}(\Omega)$ for some $\alpha \in (0,1)$, and $f$ satisfies~\eqref{e:landau} classically.  Each of $T$ and $\|f\|_{L^{\infty,k}([0,T]\times \R^6)}$ depend only on $\gamma$, $K$, $\delta$, $r$, and $R$, and $\alpha$ depends on the same quantities plus $\Omega$.

	
	\item {\bf (Matching with initial data)} For any compact $K_x \subset \R^3$ and any $\phi(t,x,v)$ satisfying, for some $\eta > 0$, 
	\[\begin{split}
	&~~\phi \in C_t L_{x,v}^{2,-7/2+\eta}([0,T]\times K_x\times\R^3)),
	\quad (\partial_t + v\cdot\nabla_x)\phi \in L^{2,-7/2+\eta}([0,T]\times K_x\times \R^3)\\
	&~~\nabla_v \phi \in L^{2,1+\gamma/2}([0,T]\times K_x \times \R^3)
	\quad \text{and} \quad \supp(\phi) \subset [0,T)\times K_x \times \R^3,
	\end{split}\]
	the solution $f$ constructed in (i) satisfies
	\begin{equation}\label{e:weak_equation}
		\begin{split}
			&\int_{\R^6} f_{\rm in}(x,v) \phi(0,x,v) \dd x \dd v\\
				&\qquad = \int_{[0,T]\times \R^6} \left(f (\partial_t + v\cdot \nabla_x) \phi - \nabla_v \phi \cdot( \bar a^{f} \nabla_v f) - f \bar b^{f} \cdot \nabla_v \phi \right) \dd x \dd v \dd t.
		\end{split}
	\end{equation}
	\end{enumerate}

	\begin{enumerate}
	\item[(iii)] {\bf (Higher regularity)} Let $T_E$ be the maximal time of existence of the solution constructed above, and let $\bar T \in (0,T_E]$.  For any partial derivative $\partial_t^j \partial_x^\beta \partial_v^\eta$, there exists $k_M >0$ and $\ell_M$ depending only on $\Psi(\bar T)$, $\bar T$, $\gamma$, $\delta$, $r$, $R$, and $M = 2j + 3|\beta| + |\eta|$, such that if $f_{\rm in} \in L^{\infty,k_M}(\R^6)$, then 
	\[
		\partial_t^j \partial_x^\beta \partial_v^\eta f \in L^{\infty,k_M - \ell_M}([0,\bar T]\times \R^6)
	\]
	and is continuous for all positive times.
	\end{enumerate}
\end{theorem}

A few comments are in order.  Firstly, if the initial data is continuous, we can show that $f$ does match the initial data in the pointwise sense (see Proposition \ref{p:initial}).  Secondly, by ``classical'' solutions, we mean those elements of $C^{2,\alpha}_{\rm kin, loc}$ that satisfy~\eqref{e:landau} pointwise; this does not necessitate that $\partial_t f$ and $\nabla_x f$ exist pointwise but rather that $(\partial_t + v\cdot\nabla_x)f$ does (see \Cref{sec:notation} for a discussion of the kinetic H\"older spaces).  Lastly, we point out that \Cref{t:exist}.(iii) implies that if $f_{\rm in} \in L^{\infty,k}$ for all $k$, then $f$ is smooth and has infinitely many moments, that is, $\partial_t^j \partial_x^\beta \partial_v^\eta \in L^{\infty,k}$ for any $j$, $\beta$, $\eta$, and $k$, for as long as $\Psi$ remains bounded.

Next, we extend our continuation criterion from \cite{HST2018landau} to the solutions of Theorem \ref{t:exist}, and sharpen it in the case of very soft potentials:
\begin{theorem}\label{t:continuation}
The solution $f$ constructed in Theorem \ref{t:exist} can be extended for as long as the quantity $\Psi$ remains finite, where $p>3/(3+\gamma)$ and $p=\infty$ for $\gamma = -3$.   More formally, $T_E = \sup\{t \geq 0 : \Psi(t) < \infty\}$, where $T_E$ is defined in \Cref{t:exist}.
\end{theorem} 
  

Because the solution constructed in Theorem \ref{t:exist} lies in a weak space relative to the order of the equation, uniqueness is a challenging issue, and in fact our proof requires stronger assumptions on $f_{\rm in}$: H\"older continuity and a lower bound assumption that rules out vacuum regions in $x$ (see \Cref{s:ideas} below for an explanation of these extra hypotheses). The uniqueness or non-uniqueness of the solutions of Theorem \ref{t:exist}, without any extra assumptions on $f_{\rm in}$, remains an open question.

In nonlinear equations where the initial data has low regularity, it often more difficult to prove uniqueness than existence of solutions even when they smooth immediately.  For another example of such a situation, where uniqueness has not been established without extra hypotheses, even though the system regularizes instantantly, see the work of Kiselev-Nazarov-Shterenberg \cite{kiselev2008blowup} on the fractal Burgers equation (i.e., Burgers equation with fractional dissipation).

\begin{theorem}\label{t:unique}
Assume that $f_{\rm in}$ satisfies:
	\begin{itemize}
		\item $f_{\rm in} \in C^\alpha$ for some $\alpha \in (0,1)$, and $f_{\rm in} \in L^{\infty,k}$ for some $k$;
		\item there exist $r,\delta, R>0$ such that for each $x\in \R^3$, there exists $|v_x|<R$ such that $f_{\rm in} \geq \delta \1_{B_r(x,v_x)}$.
	\end{itemize}
	Let $f$ be the solution constructed in \Cref{t:exist}, and let $T>0$ be any time such that $\Psi(T) < \infty$.  Then there exists $k_\alpha$ and $T_H\in (0,T]$ such that if $k \geq k_\alpha$, then $f\in\Ckin^\alpha([0,T_H]\times \R^6)$ and, for any uniformly continuous $g\in L^{\infty,5+\gamma+\eta}([0,T_H]\times \R^6)$ with $\eta >0$ such that $g$ solves \eqref{e:landau} weakly and $g(t,x,v) \to f_{\rm in}$ as $t\to 0+$, we have $f=g$. 
	 Moreover, if $f_{\rm in} \in L^{\infty,k}$ for all $k \geq 0$, then the above holds with $T_H = T$.  We have that $k_\alpha$ depends only on $\alpha$ and $\gamma$ and $T_H$ depends only on $r$, $\delta$, $R$, $k$, $\alpha$, $\gamma$, $\|f_{\rm in}\|_{L^{\infty,k}}$, and $\|f_{\rm in}\|_{C^\alpha}$
\end{theorem}

Let us make some brief comments on \Cref{t:unique}.  
First, our notion of weak solution is made precise in Section \ref{s:unique} (see the comment after Proposition \ref{prop:uniqueness_small_times}). Second, our assumptions on $f_{\rm in}$ imply, via interpolation between $C^\alpha$ and $L^{\infty,k}$, that 
\[
	\vv^m f \in C^\beta(\R^6) \qquad \text{if } m \leq k\left(1 - \frac{\beta}{\alpha}\right),
\]
see \Cref{lem:holder_decay}. For the sake of convenience, we prove that the conclusion of Theorem \ref{t:unique} holds under the assumption that $f_{\rm in}\in L^{\infty,k}(\R^6)$ and $\vv^m f_{\rm in} \in C^\alpha(\R^6)$ for $m$ and $k$ sufficiently large, depending on $\alpha$ and $\gamma$ (see Proposition \ref{prop:holder_propagation}), which, by the above, is enough to establish \Cref{t:unique}. 
Finally, we note that the positivity condition, that is, the existence of $\delta$, $r$, and $R$, is met in many standard cases such as for initial data that is continuous, periodic, and positive.

\subsection{Conditional regularity and continuation}\label{s:conditional}

This paper fits within the program of seeking weaker conditions under which solutions of the Landau equation remain smooth and can be extended past a given time. To describe this recent thread of research in more detail, let us recall the following hydrodynamic quantities associated to the solution $f$:
\begin{align*}
	M_f(t,x) &= \int_{\R^3} f(t,x,v)\dd v, &&\mbox{(mass density)}\\
	E_f(t,x) &= \int_{\R^3} |v|^2 f(t,x,v)\dd v, &&\mbox{(energy density)}\\
	H_f(t,x) &= \int_{\R^3} f(t,x,v)\log f(t,x,v) \dd v. &&\mbox{(entropy density)}
\end{align*}
We note that, in the homogeneous (i.e., $x$-independent) setting, $M_f(t)$ and $E_f(t)$ are conserved and $H_f(t)$ is non-increasing, and this behavior is crucial to the well-posedness theory.

In \cite{golse2016}, Golse-Imbert-Mouhot-Vasseur established local H\"older regularity for solutions of~\eqref{e:landau} with $M_f$ bounded above and away from zero, and $E_f$, $H_f$, and $\|f(t,x,\cdot)\|_{L^\infty_v}$ bounded above, uniformly in $t$ and $x$. Under these conditions, the coefficients in \eqref{e:coeffs} are all bounded above, and $\bar a^f$ is uniformly elliptic. The equation can then be treated as a linear kinetic Fokker-Planck equation with bounded measurable coefficients, and the main contribution of \cite{golse2016} was in adapting De Giorgi's method to such linear equations (see also \cite{pascucci2004ultraparabolic, Wang2011ultraparabolic}.)  In \cite{cameron2017landau}, Cameron-Silvestre-Snelson derived a conditional global $L^\infty$ estimate in terms of the upper bounds for $M_f$, $E_f$, and $H_f$, and the lower bound for $M_f$, in the case $\gamma \in (-2,0]$ (see \cite{snelson2018hardpotentials} for a similar result for $\gamma \in (0,1]$).  Next, Henderson-Snelson \cite{henderson2017smoothing} established $C^\infty$ regularity in the same conditional regime, with the additional assumption that infinitely many $L^1$ moments of $f$ in $v$ are finite. This study also relied on the linear theory via Schauder estimates for kinetic equations, but made essential use of the coupling between $f$ and the coefficients in the bootstrapping procedure.

Most recently, in \cite{HST2018landau} the present authors combined local well-posedness with these conditional results to derive a continuation criterion and improved this criterion by establishing self-generating lower bounds for $f$ that suffice to remove the assumptions that $M_f$ is bounded below and $H_f$ is bounded above. 
For very soft potentials ($\gamma \in [-3,-2]$), the methods of \cite{henderson2017smoothing} and \cite{HST2018landau} all go through, but under stricter conditions on $f$ that, unlike the hydrodynamic quantities above, are not physically meaningful. In the above notation, the continuation criterion established in \cite{HST2018landau} was
\begin{equation}\label{e:tilde-Psi}
\tilde \Psi(t) :=  \begin{cases} \| f\|_{L^\infty_{t,x} L^{1,2}_v([0,t]\times\R^6)}, &\gamma \in (-2,0),\\
	\|f\|_{L^\infty_{t,x} L^{1,\ell}_v ([0,t]\times\R^6) }+ \|f\|_{L^\infty_{t,x,v}([0,t]\times\R^6)}, &\gamma\in [-3,-2],\end{cases}
\end{equation}
with $\ell > 3|\gamma|/(5+\gamma)$. For $\gamma \in [-3,-2]$, Theorem \ref{t:continuation} above improves this criterion in two ways. The improvement from $\|f\|_{L^\infty_{t,x,v}}$ to $\|f\|_{L^\infty_{t,x} L^p_v}$ with $p$ as in Theorem \ref{t:continuation} was attainable with the methods of \cite{HST2018landau}, but improving $\|f\|_{L^\infty_{t,x} L^{1,\ell}_v}$ to $\|f\|_{L^\infty_{t,x} L^1_v}$ requires the more general decay estimates of the current paper. Theorem \ref{t:continuation} is also an improvement over \cite{HST2018landau} because it applies to solutions with more general initial data.

The significance of lower bounds for $f$, besides implying that vacuum regions in the initial data are instantly filled, is in getting a lower ellipticity constant for $\bar a^f$, which is needed to apply local estimates (either De Giorgi or Schauder type). Taking $f_{\rm in}$ to be well-distributed (see Definition \ref{d:well}) guarantees this ellipticity constant is uniform in $x$, after a short time has passed. The extra lower bound assumption of Theorem \ref{t:unique} ensures this ellipticity constant does not degenerate as $t\to 0$. It would be interesting to remove these structural lower bounds for $f_{\rm in}$ from the results in the current paper, both for a more robust local well-posedness theory, and because doing so may lead to better understanding of continuation for states that are not controlled uniformly in $x$.  Doing so would, however, require a completely new approach that does not depend on the regularity estimates near $t=0$.

There is a parallel program of conditional regularity for the non-cutoff Boltzmann equation: see \cite{silvestre2016boltzmann, imbert2016weak, imbert2018schauder, imbert2018decay}. So far, there is no local well-posedness result for polynomially-decaying initial data to pair with these conditional estimates, and we plan to explore this question in a forthcoming article. See also the review \cite{mouhot2018review} for more on the conditional regularity of both equations.


\subsection{Related work} \label{s:work}

There are many existence and regularity results for the spatially homogeneous ($x$-independent) Landau equation, 
see \cite{arsenev-peskov, villani1998landau, desvillettes2000landau, alexandre2015apriori, Wu2014global, gualdani2014radial, gualdani2017landau, silvestre2015landau, GGIV2019partial} and the references therein. In this setting, large-data global solutions are known to exist in the cases $\gamma \in [-2,1]$, but for $\gamma \in [-3,-2)$, the problem remains open.

In the inhomogeneous setting considered here, a suitable notion of weak solution has been defined.  Global solutions in this class have been established by Villani \cite{villani1996global} for general initial data (see also Lions \cite{lions1994boltzmannlandau} and Alexandre-Villani \cite{alexandre2004landau}). The uniqueness and regularity of these solutions are not understood.

As mentioned above, there is a large literature on close-to-Maxwellian solutions of the Landau equation that exist globally and converge to equilibrium as $t\to \infty$: see for example \cite{guo2002periodic, mouhot2006equilibrium, strain2006almostexponential, strain2008exponential, carrapatoso2016cauchy,  CM2017verysoft, duan2019mild}. 
The majority of these papers work with initial data that is close to a Maxwellian in an exponentially-weighted (in $v$) norm, which implies $f_{\rm in}$ decays exponentially at worst, but the work of Carrapatoso-Mischler \cite{CM2017verysoft} improves this to a polynomially-weighted $H^2_x L^2_v$ norm (see also \cite{carrapatoso2016cauchy}, which works with polynomially-weighted norms in the case $\gamma > 0$.) In terms of regularity, early results took $f_{\rm in}$ in a high-order Sobolev space, and subsequent works gradually enlarged the allowable space.  

The recent interesting work of Duan-Liu-Sakamoto-Strain \cite{duan2019mild} constructs mild solutions with initial data close to a Maxwellian in an exponentially-weighted $L^1_k L^2_v$ space (here, $L^1_k$ refers to the Wiener algebra in $x$ which contains all $C^1$ functions but not all H\"older continuous functions) on $\T^3\times \R^3$.  When working on a bounded spatial domain instead of $\T^3$, they require an additional derivative in this space.  Uniqueness is also shown, as is convergence to the Maxwellian, which is their main interest.

A somewhat different setting was considered by Luk in~\cite{luk2019vacuum}.  For $\gamma \in (-2,0)$, he has shown the global existence of a solution with initial data close to the vacuum state $f\equiv 0$.



\subsection{Proof ideas}\label{s:ideas}

\subsubsection{Existence} 

First, let us point out some disadvantages of the usual method of $L^2$-based energy estimates like the one pursued in~\cite{HST2018landau}. 
Because of their physical relevance, we wish for our class of initial data to include (global) Maxwellians, which do not have finite $L^2_{x,v}(\R^6)$ norm.  Thus, it makes sense to estimate $\|\varphi f(t)\|_{L^2_{x,v}}$ for cut-off functions $\varphi$ in $x$.    To this end, after multiplying \eqref{e:landau} by $\varphi^2 f$ and integrating over $\R^6$, we find, after some formal computations,
\begin{equation}\label{e:energy}
	\frac 1 2 \frac d {dt} \|\varphi f\|_{L^2_{x.v}}^2 = \int_{\R^6}\left(- \frac 1 2 f^2 v\cdot \nabla_x(\varphi^2) - \varphi^2 \nabla_v f\cdot (\bar a^f \nabla_v f)  + \frac{1}{2} \bar c^f \varphi^2 f^2\right) \dd x \dd v.
\end{equation}

The first difficulty is that the $f^2 v\cdot\nabla_x(\varphi^2)$ term\footnote{If the spatial domain were the torus $\mathbb T^3$ instead of $\R^3$, no cut-off is necessary so this term disappears from the energy estimates, which is one reason the spatially periodic case is simpler.  It is important to note that this is not a purely technical issue; it is easy to construct solutions to $(\partial_t + v\cdot\nabla_x)g = 0$ on $\R^6$ such that $g(0,\cdot,\cdot) \in L^2$ but $g(t, \cdot,\cdot) \notin L^2$ for some $t>0$.  In general, transport can cause moment loss on the whole space (but not on the torus).} may be unbounded for large $|v|$.   Thus, this term cannot be controlled by the unweighted $L^2_{x,v}$ norm of $f$. In \cite{HST2018landau}, we avoid this issue by dividing the solution $f$ by a time-dependent Gaussian, i.e.\ studying the equation for $g = e^{(\rho-\kappa t)\vv^2}f$, which has an extra term of $\kappa \vv^2 g$ with the right sign to absorb the other terms with growth in $v$ in the energy estimates (this approach was applied earlier to the Boltzmann equation in \cite{amuxy2010regularizing, amuxy2011bounded, amuxy2013mild}).  However, this method requires the initial data $f_{\rm in}$ to have Gaussian decay in $v$, and we want to get around this requirement.

The second difficulty is that the coefficients are not adequately controlled by the $L^2$-norm of $f$.  For example, $\bar c^f$ clearly must be bounded in $L^\infty$ in order to make the last term in~\eqref{e:energy} finite for $\varphi f$ merely in $L^2$. Unfortunately, $\|\bar c^f\|_{L^\infty}$ is not controlled by $\|\varphi f\|_{L^2}$, and one instead needs a bound on $\|\varphi f \|_{L^\infty_xL_v^p}$ for $p\geq 2$ depending on $\gamma$.  This necessitates bounds on higher derivatives of $f$ in order to use an embedding theorem.

When searching for bounds on derivatives of $f$, we encounter another issue.  To obtain such bounds, one might be tempted to differentiate the equation and then follow a similar strategy as above to obtain estimates.  Unfortunately, 
differentiating the equation brings up new technical difficulties---for example, when derivatives fall on $\bar a^f$, positive-definiteness is lost, so the corresponding term does not have a good sign anymore.

To side-step these issues, we base our construction on an estimate in $L^{\infty,k}([0,T]\times\R^6)$ with $k$ as in Theorem \ref{t:exist}. This estimate (Lemma \ref{l:infty}) follows from a maximum principle  argument with barriers of the form $e^{\beta t} \vv^{-k}$. 
It is interesting to note that Lemma \ref{l:infty} does not rely on lower bounds for the matrix $\bar a^f$, or on the anisotropy of $\bar a^f$ in $v$. The $L^{\infty,k}$ estimate gives good upper bounds on the coefficients, which then allows us to adapt the mass-spreading theorem of \cite{HST2018landau} to get lower bounds of $f$ that imply coercivity of $\bar a^f$. 

To pass from these \emph{a priori} estimates---which require good smoothness and decay---to an existence theorem, we must approximate $f_{\rm in}\in L^{\infty,k}(\R^6)$ by smooth, compactly supported functions $f_{\rm in}^\eps$ and apply our previous existence theorem from \cite{HST2018landau} for $H^4$, rapidly-decaying data. 
%
%
Using the $L^{\infty,k}$ estimate of Lemma \ref{l:infty} and---crucially---our continuation criterion from \cite{HST2018landau}, we can extend these approximate solutions up to a time independent of $\eps$. This step is where the restriction $k > \max\{5, 15/(5+\gamma)\}$ comes from, because for such $k$, $\|f\|_{L^{\infty,k}}$ controls the quantities in the continuation criterion of \cite{HST2018landau} (see \eqref{e:tilde-Psi}). Finally, we can apply the local regularity estimates of \cite{golse2016, henderson2017smoothing} to obtain a solution $f$ in $C^{2,\alpha}_{\rm kin, loc}$ by compactness. 


\subsubsection{Uniqueness}

To demonstrate some of the difficulties in proving uniqueness in a weak space, let us consider two solutions $f$ and $g$ with the same initial data. Then $w := f-g$ satisfies
\[\partial_t w + v\cdot \nabla_x w = \tr(\bar a^g D_v^2 w) + \tr(\bar a^w D_v^2 f) + \bar c^g w + \bar c^w f.\]
Ignoring the growth of terms on the right for large $v$ (which we can deal with by multiplying $w$ by a polynomial weight), the most difficult term in this equation to bound in terms of $w$ is $\tr(\bar a^w D_v^2 f)$. With initial data only in $L^{\infty,k}$, we certainly cannot expect a uniform-in-time bound on $D_v^2 f$, but for a Gr\"onwall-style argument, an upper bound that is integrable in $t$ is good enough. Schauder estimates, along with a standard interpolation between $C^{2,\alpha}$ and $C^\alpha$ provide a bound like
\[
	\|D_v^2 f\|_{L^\infty([t/2,t]\times \R^6)}
		\lesssim  t^{-1+r(\alpha)} \|f\|_{\Ckin^\alpha([t/2,t]\times \R^6)},
\]
where $r(\alpha) = \alpha^2/(6-\alpha)$. 
(Again, we are ignoring velocity weights. See Lemma \ref{l:weighted-Schauder} for the precise statement.) 
At this point, to bound the $\Ckin^\alpha$ norm of $f$, one could try to apply the De Giorgi-type estimate of \cite{golse2016} in a kinetic cylinder of radius $\sim t^{-1/2}$ centered at each $(t,x,v)$, but the constant in this estimate degenerates like $t^{-\alpha/2}$,  giving a total decay like $D_v^2f(t,x,v) \lesssim t^{-1-\alpha/2+r(\alpha)}$, which is not integrable.  


We are therefore led to take initial data $f_{\rm in}$ that is H\"older continuous, and try to propagate the H\"older modulus forward in time. This is the subject of Section \ref{s:holder}. Our method of proof modifies an idea used in \cite{constantin2015SQG} for the forced critical SQG equation: 
for $(t,x,v,\chi,\nu) \in \R_+ \times \R^6 \times B_1(0)^2$ and $\ell > 0$, define
\begin{equation*}
g(t,x,v,\chi,\nu) = \frac{|f (t,x+\chi,v+\nu) - f(t,x,v)|^2}{(|\chi|^{2} + |\nu|^2)^{\alpha}} \vv^{\ell}.
\end{equation*}
The function $g$ is chosen so that the size of $g$ in $L^\infty_{x,v,\chi,\nu}$  controls the weighted $C^\alpha_{x,v}$ norm\footnote{in the Euclidean H\"older metric, i.e. without the kinetic scaling of $\rho(z,z')$. This choice is imposed on us by the proof---the reason is explained in Section \ref{s:holder}.} of $f$. 
The factor $\vv^\ell$ is there to account for polynomial decay in the H\"older modulus for large $|v|$. Calculating the equation satisfied by $g$, it can be shown that all the terms either respect a maximum principle, or can be bounded by a constant times $g$, so the $L^\infty$ norm of $g$ is bounded for some positive amount of time. Since this argument only gives H\"older continuity in $x$ and $v$ for $f$, we also have to show, via the equation, that this implies H\"older continuity in all three variables. (See Appendix \ref{s:reg}.)  This argument provides a method to establish H\"older continuity of $f$---in the case that $f_{\rm in}$ happens to be H\"older continuous---without appealing to the De Giorgi estimate of \cite{golse2016}.

\subsection{Notation}\label{sec:notation}

To aid the reader, whenever possible we denote supersolutions with overlines and subsolutions with underlines, e.g. $\overline f$ and $\underline f$.

We often use $z$ to refer to a point $(t,x,v) \in [0,T]\times \R^3 \times \R^3$, and if $z$ is decorated by a symbol the coordinates are as well, e.g.\ $\tilde z = (\tilde t, \tilde x, \tilde v)$.  

For $z$, $z'$, define the kinetic distance
\begin{equation}\label{e:rho}
	\rho(z,z') = |t'-t|^{1/2} + |x'-x - (t'-t)v|^{1/3} + |v'-v|.
\end{equation}
It is not a metric since the triangle inequality is not satisfied and it is not necessarily symmetric; however, it is straightforward to check that $\rho(z_1,z_2) \lesssim \rho(z_2,z_1)$ for any $z_1$ and $z_2$.  

The kinetic distance $\rho$ gives rise to kinetic H\"older norms, which give rise to the kinetic H\"older spaces in the obvious way.  For $Q\subset [0,T]\times \R^6$, we define the H\"older seminorm
\begin{equation}\label{e:holder_seminorm}
	[u]_{\Ckin^\alpha(Q)} =  \sup_{z\neq z' \in Q} \frac{ |u(z) - u(z')|}{\rho(z,z')^{\alpha}},
\end{equation}
and the norm $\|u\|_{\Ckin^\alpha(Q)} := [u]_{\Ckin^\alpha(Q)} + \sup_Q |u|$. In addition, we define the second order norm
\[
	\|u\|_{\Ckin^{2,\alpha}(Q)} := \sup_Q |u| + \sup_Q |\nabla u| + \|D_v^2 u\|_{\Ckin^\alpha(Q)}  + \|(\partial_t + v\cdot \nabla_x ) u\|_{\Ckin^\alpha(Q)} < \infty.
\]
The differential operator $\partial_t + v\cdot \nabla_x$ has been extended to the space $\Ckin^{2,\alpha}(Q)$ by density, and even though $(\partial_t + v\cdot \nabla_x) u$ is continuous, $\partial_t u$ and $v\cdot \nabla_x u$ need not exist in a classical sense.

%
%
%
%

We denote by $Q_r(z_0)$ a ball under $\rho$, i.e. a kinetic cylinder:
\[
	Q_r(z_0) = \{z \in (-\infty, t_0]\times \R^6 : \rho(z_0, z) < r\}.
\]
Notice that this includes only times {\em before} $t_0$.  If the center point $z_0$ is omitted in the notation, it is assumed that $z_0 = 0$. We also use $B_r(p)$ to denote a ball in the standard metric.

We use the notation $A \lesssim B$ if there is a constant $C$ such that $A \leq C B$.  In each section we clarify the dependencies of $C$, but in general, $C$ may depend on $\gamma$, $\delta$, $r$, $R$, and $k$.  We use $A \approx B$ if $A \lesssim B$ and $B\lesssim A$.

 \subsection{Outline} The rest of the paper is organized as follows. In Section \ref{s:exist}, we establish the existence of solutions (Theorem \ref{t:exist}) and prove our continuation criterion (Theorem \ref{t:continuation}). In Section \ref{s:initial}, we prove that $f$ is continuous up to $t=0$, as long as $f_{\rm in}$ is continuous. In Section \ref{s:holder}, we show propagation of H\"older regularity and a time-integrable bound for $D_v^2f$, and in Section \ref{s:unique}, we prove uniqueness of solutions (Theorem \ref{t:unique}). In Appendix \ref{s:reg}, we prove that H\"older continuity in $x$ and $v$ implies H\"older continuity in $t$, and in Appendix \ref{s:interp}, we prove some interpolation lemmas.

  \section{Existence}\label{s:exist}
 
This section references various results and proof techniques (sometimes with minor modifications) from \cite{cameron2017landau}, \cite{henderson2017smoothing}, and \cite{HST2018landau}, which are previous works on the Landau equation involving the present authors. We do this to expedite the presentation of this section, and to emphasize the more novel methods developed in the rest of the paper.  For the reader's convenience, we outline these omitted proofs when possible.

 \subsection{Coefficient bounds}

\begin{lemma}\label{l:abc} If $k > \gamma + 5$, then, with all norms over $\R^3$, we have
\begin{equation*}
\begin{split}
\bar a^f(t,x,v) &\lesssim \vv^{(\gamma+2)_+} \|f(t,x,\cdot)\|_{L^{\infty,k}_v}, \\
|\bar b^f(t,x,v)| 
&\lesssim \vv^{(\gamma+1)_+}\|f(t,x,\cdot)\|_{L^{\infty,k}_v},\\
\bar c^f(t,x,v)
&\lesssim \begin{cases}\|f(t,x,\cdot)\|_{L^\infty_v}, &\gamma = -3,\\
\|f(t,\cdot,\cdot)\|_{L^{\infty,k}_v}, &\gamma \in (-3,0),\end{cases}
\end{split}
\end{equation*}
whenever the right-hand sides are finite, with implied constants depending on $\gamma$ and $k$.
\end{lemma}
\begin{proof}
It is elementary to show that for $r >-3$ and $g:\R^3\to \R_+$, there holds for $v\in \R^3$,
\[ (g * |\cdot|^r)(v) \lesssim \vv^{r_+} \|g\|_{L^{\infty,k}},\]
where $k > r+3$ and $r_+ = \max\{r,0\}$.  The statement of the lemma follows from this convolution estimate and the formulas for $\bar a^f$, $\bar b^f$, and $\bar c^f$ in \eqref{e:coeffs}.
\end{proof}

For the proof of the continuation criterion, we also require bounds on the coefficients in terms of $L^1_v$-based norms:

 
\begin{lemma}\label{l:mass-energy}
For $f$ such that the right-hand sides are finite, we have, with all norms over $\R^3$,
\begin{align*}
|\bar a^f(t,x,v)| &\lesssim C\begin{cases} \vv^{\gamma+2}\|f(t,x,\cdot)\|_{L^{1,2}_v}, &\gamma \in (-2,0),\\
\|f(t,x,\cdot)\|_{L^{3/(5+\gamma)+\eta}_v} + \|f(t,x,\cdot)\|_{L^1_v}, &\gamma \in [-3,-2],\end{cases}\\
|\bar b^f(t,x,v)| &\lesssim\begin{cases} \vv^{\gamma+1}\|f(t,x,\cdot)\|_{L^{1,1}_v}, &\gamma \in (-1,0),\\
\|f(t,x,\cdot)\|_{L^{3/(4+\gamma)+\eta}_v} + \|f(t,x,\cdot)\|_{L^1_v}, &\gamma \in [-3,-1],\end{cases}\\
\bar c^f(t,x,v)| &\lesssim \begin{cases}\|f(t,x,\cdot)\|_{L^{3/(3+\gamma)+\eta}_v} + \|f(t,x,\cdot)\|_{L^1_v}, &\quad \gamma \in (-3,0),\\
\|f(t,x,\cdot)\|_{L^\infty_v}, &\quad \gamma = -3,\end{cases}
\end{align*}
with $\eta>0$ any small constant. 
If $f$ has even more decay, then the quadratic form associated to $\bar a^f$ satisfies the following anisotropic upper bounds for $e\in \mathbb S^2$: If $\gamma \in [-3,-2]$, then
\begin{equation}\label{e:anisotropic}
 \bar a^f(t,x,v) e_i e_j \leq C\left(\|f(t,x,\cdot)\|_{L^{1,\ell}_v} + \|f(t,x,\cdot)\|_{L^p_v}\right) \begin{cases} \vv^{\gamma+2}, & e \perp v,\\
\vv^\gamma, & e \parallel v, \end{cases}
\end{equation}
where $p > 3/(5+\gamma)$ and  $\ell> 3|\gamma|/(5+\gamma)$. If $\gamma \in (-2,0)$, then the same estimate holds without the $L^p_v$ bound, and with $2$ replacing $\ell$.
\end{lemma}
\begin{proof}
The isotropic upper bounds follow from standard integral estimates and are omitted. For the proof of \eqref{e:anisotropic}, see \cite[Lemma 2.1]{cameron2017landau} for the case $\gamma \in (-2,0)$ and \cite[Lemma A.3]{henderson2017smoothing} for the case $\gamma \in [-3,-2]$. The statement of \cite[Lemma A.3]{henderson2017smoothing} uses the $L^\infty_v$ norm rather than the $L^p_v$ norm, but the same proof works with $L^p_v$.
\end{proof}

\subsection{A priori estimates}

Our first step is to get a closed estimate in the space $L^{\infty,k}(\R^6)$ for suitable $k$, using the maximum principle for the linear Landau equation.


\begin{lemma}\label{l:infty}
Let $k_0 >  \gamma + 5$. If $f$ is a smooth solution to \eqref{e:landau} on $[0,T]\times \R^6$, $\|f_{\rm in}\|_{L^{\infty,k_0}(\R^6)} + \|f\|_{L^{\infty,k_0}([0,T]\times \R^6)} < \infty$, and $f \in C_t L^{\infty,k_0}_{\R^6}([0,T]\times\R^6)$, then there exists $T_0>0$ depending on $\|f_{\rm in}\|_{L^{\infty,k_0}(\R^6)}$ and $\gamma$, and $C > 0$ depending only on $\gamma$, such that 
\[ \|f(t,\cdot,\cdot)\|_{L^{\infty,k_0}(\R^6)}  \leq C\|f_{\rm in}\|_{L^{\infty,k_0}(\R^6)}, \quad 0\leq t\leq  \min(T,T_0). \]
Furthermore, if $\|f_{\rm in}\|_{L^{\infty,k}(\R^6)} < \infty$ for any  $k \geq k_0$, and either $\|f\|_{L^{\infty,k_0}([0,T]\times \R^6)}$ or $\Psi(t) + \|f\|_{L^\infty_{t,x}L^p_v([0,T]\times \R^6)}$ is finite  (recall the definition of $\Psi$ in~\eqref{e:Psi}), with $p> 3/(3+\gamma)$, 
then the inequality
\[ \|f(t,\cdot,\cdot)\|_{L^{\infty,k}(\R^6)} \leq \|f_{\rm in}\|_{L^{\infty,k}(\R^6)} e^{CK t}, \quad 0\leq t\leq T,\]
holds for either $K = \|f\|_{L^{\infty,k_0}([0,T]\times\R^6)}$ or $K = \Psi(t) + \|f\|_{L^\infty_{t,x}L^p_v([0,T]\times \R^6)}$, where $C$ depends on $\gamma$ and $k$.
\end{lemma}
Although this estimate depends on $\|f\|_{L^\infty_{t,x}L^p_v([0,T]\times \R^6)}$ for all $\gamma \in [-3,0)$, this dependence is removed in the case $\gamma \in (-2,0)$ during the proof of the continuation criterion (Theorem \ref{t:continuation}).
\begin{proof}
Define the linear operator $L$ by
\[
	Lg = \partial_tg +v\cdot \nabla_xg - \tr(\bar a^f D_v^2 g) - \bar c^f g.
\]
With $\beta >0$ and $k >0$ to be chosen, define $\overline \phi(t,x,v) = e^{\beta t}\vv^{-k}$. Then
\[|\partial_{ij} \overline\phi| = e^{\beta t} \left|k(k+2)\vv^{-k-4} v_iv_j - k \vv^{-k-2}\delta_{ij}\right| \lesssim \vv^{-2} \overline\phi,\]
and 
\begin{equation*}
\begin{split}
L\overline\phi &= \beta \overline\phi - \tr(\bar a^f D_v^2 \overline\phi) - \bar c^f \overline\phi\\
 &\geq \beta \overline\phi -  CK\vv^{(\gamma+2)_+} \vv^{-2}\overline\phi - CK \overline\phi\\
 &\geq (\beta - C_0K) \overline\phi,
\end{split}
\end{equation*}
where $K$ is any quantity such that $|\bar a^f|\leq K\vv^{(\gamma+2)+}$ and $|\bar c^f|\leq K$, and $C_0$ depends on $\gamma$ and $k$. With $\beta = C_0K$, we have $L\overline\phi \geq 0$. If $k$ is such that $\|f_{\rm in}\|_{L^{\infty,k}(\R^6)} < \infty$, we can apply the maximum principle to $\|f_{\rm in}\|_{L^{\infty,k}}\overline\phi - f$ and conclude 
\begin{equation}\label{e:Linfty}
 \|f(t,\cdot,\cdot)\|_{L^{\infty,k}(\R^6)} \leq \|f_{\rm in}\|_{L^{\infty,k}(\R^6)} \exp(C_0 K t), \quad t\in [0,T].
 \end{equation}
 Now we set $k = k_0$, and for $t\in [0,T]$, define $H(t) = \|f\|_{L^{\infty,k_0}([0,t]\times \R^6)}$. Since $H$ is increasing, \eqref{e:Linfty} applied on $[0,t]$ with $K = H(t)$ implies $H(t) \leq \|f_{\rm in}\|_{L^{\infty,k_0}}\exp(C_0tH(t))$ for $t\in [0,T]$. From this inequality and Lemma \ref{l:annoying} below, we have $H(t) \leq C\|f_{\rm in}\|_{L^{\infty,k_0}}$ if $C_0 \|f_{\rm in}\|_{L^{\infty,k_0}} t \leq 1/e$, and we can choose $T_0 = (e C_0 \|f_{\rm in}\|_{L^{\infty,k_0}})^{-1}$. 
 
For the second conclusion of the lemma, we apply \eqref{e:Linfty} for any $k \geq k_0$. We may choose $K = \|f\|_{L^{\infty,k_0}([0,T]\times \R^6)}$ by Lemma \ref{l:abc}, or $K = \Psi(t) + \|f\|_{L^\infty_{t,x}L^p_v([0,T]\times \R^6)}$ by Lemma \ref{l:mass-energy}.
\end{proof}

\begin{lemma}\label{l:annoying}
	If $H:[0,T]\to \R_+$ is a continuous increasing function and $H(t) \leq A e^{BtH(t)}$ for all $t\in [0,T]$ and some positive constants $A$ and $B$, then 
	\[H(t) \leq eA \quad \mbox{ for } \quad  0\leq t \leq T_* := \min\left(T, \frac 1 {eAB}\right).\]
\end{lemma}
\begin{proof}
	First, we may assume that $T_* = 1/(eAB)$ by simply extending $H(t)$ to be constant after $t = T$.  For each $t\in (0,  1/(eAB))$, let $\phi_t(x) = A e^{Bt x} - x$ and let $x_{\rm min}(t)  = -(Bt)^{-1} \log(ABt)$.  
	A simple computation implies that $\phi_t(x_{\rm min}(t)) < 0$. 
	We claim that $H(1/(eAB)) \leq x_{\rm min}(1/(eAB))$, which establishes the claim because $H$ is increasing.  We argue by contradiction supposing $H(1/(eAB)) > x_{\rm min}(1/(eAB))$.  Since $H(0) < x_{\rm min}(0)$, the intermediate value theorem implies that $H(t_0) = x_{\rm min}(t_0)$ for some $t_0 < 1/(eAB)$, which implies that $\phi(H(t_0)) = \phi(x_{\rm min}(t_0)) < 0$.  This is a contradiction since $\phi_t(H(t)) \geq 0$ for all $t\in [0,T_*]$ by hypothesis.  This finishes the proof.
%
\end{proof}

We require the following result guaranteeing a lower ellipticity bound for $\bar a^f$ that is controlled from below whenever $\bar a^f$ is bounded above.  \Cref{l:lower-bounds} follows from the work in \cite{HST2018landau}, although it is not explicitly stated as a lemma in that paper. We require the following definition: for $k\geq 0$, the uniformly local Sobolev norm $H^4_{\rm ul}$ is defined by
\[\|g\|_{H^{k}_{\ul}(\R^6)}^2 = \sum_{|\alpha|+|\beta|\leq k} \sup_{a\in \R^3}  \int_{\R^6} |\phi(x-a)  \partial_x^\alpha \partial_v^\beta g(x,v)|^2 \dd x \dd v,\]
where $\phi\in C^\infty_0(\R^3)$ is a cutoff satisfying $0\leq \phi \leq 1$, $\phi \equiv 1$ in $B_1$, and $\phi\equiv 0$ in $\R^3\setminus B_2$. A bound on $f$ in the $H^4_{\rm ul}$ norm is needed to apply the results of \cite{HST2018landau} directly, but none of the conclusions depend quantitatively on this norm.

\begin{lemma}\label{l:lower-bounds}
	Let $f\geq 0$ be a solution of \eqref{e:landau} in $[0,T]\times \R^6$ with initial data $f_{\rm in}$ well-distributed with parameters $\delta, r, R >0$, and such that $e^{\rho|v|^2}f_{\rm in}\in H^4_{\rm ul}(\R^6)$ for some $\rho>0$. 
If $K$ is a constant such that $|\bar a^f(t,x,v)| \leq K \vv^{(\gamma+2)_+}$ in $[0,T]\times \R^6$, then $f$ satisfies the pointwise lower bound
	\[f(t,x,v) \geq c_1(t) e^{-|v|^{2-\gamma}/c_1(t)}, \quad (t,x,v) \in E,\]
	 where $c_1(t)>0$ depends only on $\gamma$, $\delta$, $r$, $R$, and $K$. Furthermore, the matrix $\bar a^f$ satisfies 
	\begin{equation}\label{e:lambda}
	\bar a_{ij}^f(t,x,v) e_i e_j \geq c_2(t) \begin{cases} \vv^{\gamma},  &e\in \mathbb S^2,\\
	\vv^{\gamma+2}, &e\cdot v = 0,\end{cases} 
	\end{equation}
	for $0\leq t\leq T$, with $c_2(t)$ depending on $\gamma$, $\delta$, $r$, $R$, and $K$. The functions $c_1(t)$ and $c_2(t)$ may degenerate as $t\to 0$ but are otherwise uniformly positive.
	
	
	If, in addition, $f_{\rm in}$ is such that for all $x$, there is some $|v_m|<R$ such that $f_{\rm in}(\cdot,\cdot) \geq \delta \1_{B_r(x,v_m)}$, then \eqref{e:lambda} holds with $c_2(t)$ replaced by a uniform positive constant $c_T>0$ depending on $\gamma$, $\delta$, $r$, $R$, $K$, and $T$, that are positive for any $T< \infty$. 
\end{lemma}
\begin{proof} 

The first statement follows directly from \cite[Theorem 1.3]{HST2018landau}. As stated, that theorem requires a bound on the hydrodynamic quantities (precisely $\tilde \Psi(t)$ defined in \eqref{e:tilde-Psi}), but the only role these quantities play in the proof is providing an upper bound of the form $|\bar a^f(t,x,v)| \leq K\vv^{(\gamma+2)_+}$, so any such $K$ suffices (see \cite[Section 4]{HST2018landau}).

The lower ellipticity bound \eqref{e:lambda} for $\bar a^f$ follows from the lower bound for $f$ and \cite[Lemma 4.3]{HST2018landau}.

The last statement follows from an examination of  the proof of \cite[Proposition 4.1]{HST2018landau}. For any $x\in \R^3$, let $v_m$ be such that $f_{\rm in}(x,\cdot)\geq \delta \1_{B_r(v_m)}$, where $|v_m|<R$. Step 1 of the proof of \cite[Proposition 4.1]{HST2018landau} establishes the existence of some $T_*\in (0,T]$ depending on $\gamma$, $K$, and $r$, such that for all $t\in [0,T_*]$, 
\[f(t,\cdot,\cdot) \geq \frac \delta 2 \1_{B_{r/2}(x)\times B_{r/2}(v_m)}.\]
The proper $v$-dependence for $t\in [0,T_*]$ is then implied by the proof of \cite[Theorem 1.3(ii)]{HST2018landau}. We can take $c_1$ to be the minimum of $\delta/2$ and $\inf_{t\in [T_*,T]} c_1(t)$, and similarly for $c(t)$ in \eqref{e:lambda}.
%
%
\end{proof}

\begin{remark}
It is seen from the proof of Theorem \ref{t:exist}.(i) that the quantitative lower bounds of Lemma \ref{l:lower-bounds} also apply to the solution constructed in Theorem \ref{t:exist}.
\end{remark}

\subsection{A convenient transformation and new coefficients}

Since the ellipticity ratio of $\bar a^f$ degenerates as $|v|\to\infty$, it is convenient to use a change of variables developed in \cite[Lemma 4.1]{cameron2017landau} that makes $\bar a^f$ uniformly elliptic (see also \cite[Lemma 3.1]{henderson2017smoothing} for the extension to the case $\gamma < -2$).  This makes it possible to apply the local regularity estimates from \cite{golse2016,henderson2017smoothing} and understand precisely how the constants degenerate for large $|v|$.  We define this transformation here.

Fix $z_0 \in \R_+\times \R^6$, and let $S$ be the linear transformation such that 
\begin{equation*}
S e =
	\begin{cases}
		\vvO^{1+\gamma/2} e , & e \cdot v_0 = 0\\
		\vvO^{\gamma/2}e, & e \cdot v_0 = |v_0|,
	\end{cases}
\end{equation*}
and
\begin{equation}\label{e:r_1}
	r_1 = \vvO^{-(1+\gamma/2)_+}
		\min\left(1,\sqrt{t_0/2}\right).
\end{equation}
Next, define
\begin{equation}\label{e:COV}
	\mathcal S_{z_0}(t,x,v) = (t_0+t,x_0+S x + t v_0 ,v_0 + S v)
		\quad \text{ and } \quad
	\delta_{r_1}z = (r_1^2 t, r_1^3 x, r_1 v).
\end{equation}
There are two important, elementary features of these transformations that we require throughout the proof:
\begin{equation}\label{e:rho_transformation}
	\begin{split}
	&\rho(\delta_{r_1} z,\delta_{r_1} z') = r_1 \rho(z,z')
		\quad \text{ and } \quad\\
	&\min\left(1, \sqrt{t_0/2}\right) \vvO^{-(1+\gamma/2)_+ + \gamma/2} \rho(z,z')
		\leq \rho(\mathcal S_{z_0}(\delta_{r_1}z), \mathcal S_{z_0}(\delta_{r_1}z'))\\
		&\qquad \leq \min\left(1, \sqrt{t_0/2}\right) \rho(z,z').
	\end{split}
\end{equation}
Given any function $g$ and any point $z=(t,x,v) \in Q_1$, we define
\begin{equation}\label{e:transformation_L_infinity}
	g_{z_0}(z) = g(\mathcal S_{z_0}(\delta_{r_1}z)).
\end{equation}

First we prove an $L^\infty$-based bound on the coefficients. This is necessary to obtain the H\"older regularity of $f$.
\begin{lemma}\label{l:transformation_L_infinity}
	Fix $z_0 = (t_0,x_0,v_0)\in \R_+\times \R^3\times\R^3$ and $\ell \in \R$.  Let $\bar a$ be a matrix and $g$ be a scalar- or vector-valued function, and assume that, for all $z\in [t_0/2,t_0]\times \R^3\times \R^3$ and $e \in \mathbb S^2$,
	\begin{equation}\label{e:g-a}
		|g(z)| \lesssim \vv^\ell
			\quad\text{and}\quad
		\bar a_{ij}(z) e_i e_j
			\approx
				\begin{cases}
					\vv^\gamma \quad &\text{ if } |e \cdot v| = |v|,\\
					\vv^{2+\gamma} \quad &\text{ if } e\cdot v = 0.
				\end{cases}
	\end{equation}
	Let $\overline A = S^{-1} \bar a_{z_0} S^{-1}$.  Then there exist $0 < \lambda < \Lambda$, depending only on the implied constants in \eqref{e:g-a}, such that, in $Q_1$,
	\[
		\lambda I
			\leq \overline A(z)
			\leq \Lambda I
		\quad \text{and}\quad
		|g_{z_0}(z)|
			\leq \Lambda \vvO^{-\ell} \|g\|_{L^{\infty,\ell}([t_0/2,t_0]\times \R^6)}.
	\]
	In particular, $\lambda$ and $\Lambda$ are independent of $z_0$.
\end{lemma}
\begin{proof}
	The proof of the bounds of $\overline A$ is exactly as in \cite[Lemma 3.1]{henderson2017smoothing}.  The bound on $g_{z_0}$ follows directly from the fact that if $z\in Q_1$, then $\langle r_1Sv + v_0\rangle \approx \vvO$.
\end{proof}

In order to obtain bounds on the $\Ckin^{2,\alpha}$ norm of $f$, we require H\"older regularity of the coefficients of~\eqref{e:landau} after applying our transformation.
\begin{lemma}\label{l:Holder_bar_A}
	Fix $z_0 \in \R_+\times\R^6$, $\alpha\in(0,1)$, and $m > \max\{5,5 + \gamma + \alpha/3\}$.  Suppose that $\vv^m f \in \Ckin^\alpha([t_0/2,t_0]\times \R^6)$.  Let $\overline A(z) = S^{-1} \bar a^f_{z_0} S^{-1}$ and $\overline C = r_1^2 \bar c^f_{z_0}$.  Then we have
	\[\begin{split}
		 [\overline A]_{\Ckin^{2\alpha/3}(Q_1)}
			&\lesssim \vvO^{2 + \alpha/3} [\vv^m f]_{\Ckin^\alpha([t_0/2,t_0]\times \R^6)},\\
			[\overline C]_{\Ckin^{2\alpha/3}(Q_1)}
			&\lesssim \vvO^{-(2+\gamma)_+ + \gamma + \alpha/3} [\vv^m f]_{\Ckin^\alpha([t_0/2,t_0]\times \R^6)},\\
		[f_{z_0}]_{\Ckin^{2\alpha/3}(Q_1)}
			&\lesssim \min\{1,t_0^{\alpha/3}\}\vvO^{2 + \gamma + \alpha/3} [\vv^m f]_{\Ckin^\alpha(Q_1(z_0) \cap ([t_0/2,t_0]\times\R^6))}.
	\end{split}\]
\end{lemma}
\begin{proof}
	We prove only the first inequality; that is, the inequality for $\overline A$.  The bound for $\overline C$ is exactly analogous, and the bound for $f_{z_0}$ is straightforward after using~\eqref{e:rho_transformation}.
	
	The following calculation is similar to \cite[Lemma 3.3]{henderson2017smoothing}.  
For $z, z'\in Q_1$, let $\tilde z = \mathcal S_{z_0}(\delta_{r_1}z)$ and $\tilde z' = \mathcal S_{z_0}(\delta_{r_1} z')$.  For $\overline A(z) = S^{-1}\bar a^{f}(\mathcal S_{z_0}(\delta_{r_1}(z)))S^{-1}$, we have, using $|S^{-1}e|\lesssim \langle v_0\rangle^{-\gamma/2}|e|$,
\begin{equation}\label{e:Az}
\begin{split} 
	 &|\overline A(z) - \overline A(z')|
	 	\lesssim \vvO^{-\gamma}\int_{\R^3} |w|^{\gamma+2} |f(\tilde t, \tilde x, \tilde v - w) - f(\tilde t', \tilde x', \tilde v' -w)| \dd w\\
		&\qquad\lesssim \vvO^{-\gamma}  \int_{\R^3} |w|^{\gamma+2} \left(\rho(\tilde z,\tilde z')^\alpha + \rho(\tilde z,\tilde z')^{2\alpha/3} |w|^{\alpha/3}\right)  [f]_{\Ckin^\alpha([t_0/2,t_0]\times \R^3 \times B_1(v_0-w))}  \dd w.
\end{split}
\end{equation}
where we have used that
\[
	\rho((\tilde t, \tilde x, \tilde v - w), (\tilde t', \tilde x', \tilde v' - w)) \leq \rho(\tilde z,\tilde z') + \rho(\tilde z,\tilde z')^{2/3} |w|^{1/3},
\]
which can be seen by a direct computation.

The inequality $[gh]_{\Ckin^\alpha} \leq \|g\|_{L^\infty} [h]_{\Ckin^\alpha} + [g]_{\Ckin^\alpha} \|h\|_{L^\infty}$ implies 
\[
\begin{split}
	 [f]_{\Ckin^\alpha([t_0/2,t_0]\times \R^3 \times B_1(v_0-w))} 
 		&\lesssim \langle v_0-w\rangle^{-m} \|\vv^m f \|_{\Ckin^\alpha}.
 \end{split}
\]
(Here, and for the remainder of the proof, we write $\Ckin^\alpha = \Ckin^\alpha([t_0/2,t_0]\times \R^6)$.)  Feeding this estimate into \eqref{e:Az} and using~\eqref{e:rho_transformation} and that $m > \max\{3 ,5 + \gamma + \alpha/3\}$, we have
\[
	[\overline A]_{\Ckin^{2\alpha/3}(Q_1)}
		= \sup_{z\neq z'\in Q_1} \frac{|\overline A(z) - \overline A(z')|}{\rho(z,z')^{2\alpha/3}}
		\lesssim 
			\|\vv^m f \|_{\Ckin^\alpha} \langle v_0\rangle^{2 + \alpha/3}.
\]
A similar calculation, with $\gamma$ replacing $\gamma+2$, implies
\[
	[\overline C]_{\Ckin^{2\alpha/3}(Q_1)}
		\lesssim 
			\vvO^{-(2+\gamma)_+ + \gamma + \alpha/3} \|\vv^m f \|_{\Ckin^\alpha}.
\]
\end{proof}

\subsection{Existence of a solution: \Cref{t:exist}}

We are now ready to prove the existence of solutions:
\begin{proof}[Proof of Theorem \ref{t:exist}]  (i) Fix any $\e>0$ and define the following mollification and cut-off functions.  Fix any $\psi \in C_c^\infty(\R^6)$ such that $\psi \geq 0$ and $\int_{\R^6} \psi \dd x \dd v = 1$.  Let $\psi_\e = \e^{-6} \psi(x/\e, v/\e)$.  In addition, let $\zeta_\eps\in C^\infty(\R^3)$ be such that $\zeta_\eps(v) = 1$ when $|v|\leq 1/\eps$, $\zeta_\eps(v) = 0$ when $|v|\geq 1/\eps+1$, and $|\nabla_v\zeta_\eps|\lesssim 1$.

Then let
\begin{equation}\label{e:f_in}
f_{\rm in}^\eps = \zeta_\eps(v) (f_{\rm in}*\psi_\eps)(x,v).
\end{equation}
Note that $\|f_{\rm in}^\eps\|_{L^{\infty,k}} \lesssim \|f_{\rm in}\|_{L^{\infty,k}}$. The smoothed, cut-off initial condition $f_{\rm in}^\eps$ is compactly supported in $v$, smooth, and nonnegative, so $e^{\rho |v|^2}f_{\rm in}^\eps$ is in $H^4_{\rm ul}(\R^6)$ for any $\rho>0$.  Hence, we can apply \cite[Theorem 1.1]{HST2018landau} to obtain a solution $f^\eps:[0,T_\eps]\times \R^6\to \R_+$. We may assume $T_\eps$ is the maximal time of existence of the solution $f^\eps$. By the existence theorem \cite[Theorem 1.1]{HST2018landau}, we have $e^{(\rho/2)|v|^2} f^\eps \in L^\infty([0,T_\eps],H^4_{ul}(\R^6))\subset L^\infty([0,T_\eps]\times \R^6)$, which clearly implies $f^\eps \in L^{\infty,k}([0,T_\eps]\times \R^6)$. Since $k > 5$,  \Cref{l:infty} yields
\begin{equation}\label{e:kth-moment} 
\|f^\eps\|_{L^{\infty,k}([0,T_1], \R^6)} \leq C \|f^\eps_{\rm in}\|_{L^{\infty,k}} \leq C \|f_{\rm in}\|_{L^{\infty,k}}, 
\end{equation}
where $T_1 = \min(T_0,T_\eps)$, $T_0 \lesssim (\|f_{\rm in}\|_{L^{\infty,k_0}})^{-1}$, and $C$ is independent of $\eps$. 

For $\eps$ small enough (depending only on $\delta$, $R$, and $r$), we have that $f_{\rm in}^\eps(x,v)$ is well-distributed with parameters $\delta/2$, $r/2$, and $R$. Therefore, \cite[Theorem 1.5]{HST2018landau} implies $f^\eps$ can be extended for as long as 
\[\begin{cases}\|f^\eps(t,\cdot,\cdot)\|_{L^{\infty}_x L^{1,2}_v}, & \gamma\in (-2,0),\\ \|f^\eps(t,\cdot,\cdot)\|_{L^\infty_x L^{1,\ell}_v} + \|f^\eps(t,\cdot,\cdot)\|_{L^\infty_{x,v}}, &\gamma\in [-3,-2].\end{cases}\] 
remains finite, where $\ell > 3|\gamma|/(5+\gamma)$. This quantity is controlled by $\|f^\eps(t,\cdot,\cdot)\|_{L^{\infty,k}(\R^6)}$ since $k > \max\{5, 15/(5+\gamma)\}$. Therefore, \eqref{e:kth-moment} implies the maximal time of existence $T_\eps$ must be larger than $T_0$, i.e. $f^\eps$ exists for $t\in [0,T_0]$ for all $\eps$.


The bound \eqref{e:kth-moment} holds for $f^\eps$ on $[0,T_0]$, so the upper bounds for $\bar a^{f^\eps}$, $\bar b^{f^\eps}$, and $\bar c^{f^\eps}$ of Lemma \ref{l:abc} and Lemma \ref{l:mass-energy} hold independently of $\eps$, since all the relevant norms of $f$ are controlled by $\|f\|_{L^{\infty,k}}$. Since $f_{\rm in}^\eps$ is well-distributed, the smoothing theorem \cite[Theorem 1.3]{HST2018landau} implies $f^\eps$ is a $C^\infty$ classical solution of \eqref{e:landau}, with regularity estimates that may depend on $\eps$. The decay estimate \eqref{e:kth-moment} and Lemma \ref{l:lower-bounds} with $K \lesssim \|f_{\rm in}\|_{L^{\infty,k}}$ imply lower bounds for $f^\eps$ and $\bar a^{f^\eps}$ that depend only on $t$, $x$, $\delta$, $r$, and $\|f_{\rm in}\|_{L^{\infty,k}}$, but not on $\eps$, for $t\in [0,T_0]$.

Next, we want to apply local regularity estimates at any point $z_0 \in \R_+ \times \R^6$. To track the dependence of these estimates on $v$, we must use the change of variables defined in~\eqref{e:COV}.
%
Recall the definition of $f_{z_0}$ via~\eqref{e:transformation_L_infinity}. 
By \Cref{l:transformation_L_infinity} 
(which relies on our $\eps$-independent bounds for $\bar a^{f^\eps}$,  $\bar b^{f^\eps}$, and $\bar c^{f^\eps}$, most crucially the anisotropic bounds \eqref{e:anisotropic} and \eqref{e:lambda}), 
we have that $f_{z_0}^\eps$ satisfies
\begin{equation}\label{e:isotropic}
\partial_t f_{z_0}^\eps + v \cdot \nabla_x f_{z_0}^\eps = \nabla_v \cdot\left(\overline A(z)\nabla_v f_{z_0}^\eps\right) +  \overline B(z)\cdot \nabla_v f_{z_0}^\eps +  \overline C(z) f_{z_0}^\eps,
\end{equation}
which comes from~\eqref{e:landau}, and
\begin{equation}\label{e:isotropic-nondivergence}
\partial_t f_{z_0}^\eps + v \cdot \nabla_x f_{z_0}^\eps = \textup{\tr}\left( \overline A(z)D_v^2 f_{z_0}^\eps\right) +  \overline C(z) f_{z_0}^\eps,
\end{equation}
which comes from~\eqref{e:landau_nd}, in $Q_1$, with the coefficients
\begin{equation}\label{e:COV_coefficients}
	\begin{split}
		&\overline A(z) = S^{-1}\bar a^{f^\eps}(\mathcal S_{z_0}(\delta_{r_1}(z))) S^{-1},
		\qquad \overline B(z) = r_1S^{-1}\bar b^{f^\eps}(\mathcal S_{z_0}(\delta_{r_1}(z))),\\
		&\text{and} \qquad \overline C(z) = r_1^2\bar c^{f^\eps}(\mathcal S_{z_0}(\delta_{r_1}(z)))	
	\end{split}
\end{equation}
satisfying
\begin{equation}\label{e:lambdaLambda}
\lambda_{t_0} I \leq \overline A(z) \leq \Lambda I, \quad |\overline B(z)| + |\overline C(z)| \leq \Lambda,
\end{equation}
with $\Lambda$ depending only on $\|f^\eps\|_{L^{\infty,k}([0,T_0]\times \R^6)} \lesssim K$, and $\lambda_{t_0,x_0}$ depending on $K$, $t_0$, $\delta$, $r$, and $R$. The dependence on $t_0$ comes from $c(t)$ in Lemma \ref{l:lower-bounds}, which is uniformly positive on any compact subset of $(0,T]$.

The divergence-form equation \eqref{e:isotropic} allows us to apply \cite[Theorem 3]{golse2016} to $f_{z_0}^\eps$:
\begin{equation}\label{e:c1000}
	\|f_{z_0}^\eps\|_{\Ckin^\alpha(Q_{1/2})}
		\leq C(\|f_{z_0}^\eps\|_{L^2(Q_1)} + \|\overline C f_{z_0}^\eps\|_{L^\infty(Q_1)})
		\lesssim \langle v_0\rangle^{-k},
\end{equation}
with implied constant depending on $\lambda_{t_0}$ and $K$. The H\"older exponent $\alpha\in (0,1)$ also depends on $\lambda_{t_0}$ and $\Lambda$, and therefore on $K$. 
Undoing this change of variables and using~\eqref{e:rho_transformation}, we find that
\[
\begin{split}
	\|f^\eps\|_{\Ckin^\alpha(Q_{r_1/2}(z_0))}
		&\leq \min\{1,t_0\}^{-\alpha/2}\vvO^{\alpha((1+\gamma/2)_+-\gamma/2)}\|f_{z_0}^\eps\|_{\Ckin^\alpha(Q_{1/2})}\\
		&\lesssim \min\{1,t_0\}^{-\alpha/2}\vvO^{-k+\alpha((1+\gamma/2)_+-\gamma/2)},
		\end{split}
\]
where $r_1$ is defined in \eqref{e:r_1}. Applying the straightforward interpolation
\[
	\|g\|_{\Ckin^\alpha(Q_1(z_1))}
		\leq r^{-\alpha} \|g\|_{L^\infty(Q_1(z_1))} + \sup_{z_2 \in Q_1(z_1)} \|g\|_{\Ckin^\alpha(Q_r(z_2))}
\]
for any $g$, $z_1$, $z_2$, and $r$, 
we deduce
\begin{equation}\label{e:prelim_Holder}
	\|f^\eps\|_{\Ckin^\alpha(Q_1(z_0)\cap([0,T]\times \R^6))}
		\lesssim \min\{1, t_0\}^{-\alpha/2} \vvO^{-k + \alpha((1+\gamma/2)_+ -\gamma/2)}.
\end{equation}
Since it is not an important point in this proof, we absorb all dependence on $t_0$ into the implied constant for the remainder of this section.

Next, we pass this regularity to $\overline A$ and $\overline C$ via \Cref{l:Holder_bar_A}, 
which requires $[\vv^m f^\eps]_{\Ckin^\alpha([t_0/2,t_0]\times\R^6)} \leq C_{t_0}$, where $m > \max\{5,5 + \gamma + \alpha/3\}$.  By assumption $k > 5$ so this holds with $m = k - \alpha((1+\gamma/2)_+ - \gamma/2)$, up to decreasing $\alpha$.  Thus,
\[
	[\overline A]_{\Ckin^{2\alpha/3}(Q_1)}
		\lesssim \langle v_0\rangle^{2 + \alpha/3},
	\quad
	[\overline C]_{\Ckin^{2\alpha/3}(Q_1)}
		\lesssim \langle v_0\rangle^{-(2+\gamma)_+ + \gamma + \alpha/3},
\]
with constants depending on $t_0$ and $K$.  It is then straightforward to show that
\[
\begin{split}
	&\|\overline A\|_{\Ckin^{2\alpha/3}(Q_1)}
		\lesssim \langle v_0 \rangle ^{2+\alpha/3}, \text{ and }\\
	&[\overline C f_{z_0}^\e]_{\Ckin^{2\alpha/3}(Q_1)}
		\lesssim \vvO^{-k - (2+\gamma)_+ + \gamma + \alpha/3}.
\end{split}
\]
Now, using the non-divergence form equation \eqref{e:isotropic-nondivergence}, we can apply the Schauder-type estimate \cite[Theorem 2.9]{henderson2017smoothing} to $f_{z_0}^\eps$:
\begin{equation}\label{e:Dv2f}
\begin{split}
	[f_{z_0}^\eps]_{\Ckin^{2,2\alpha/3}(Q_{1/2})} 
		&\lesssim [\overline C f_{z_0}^\eps]_{\Ckin^{2\alpha/3}(Q_1)}
			+ \|\overline A\|_{\Ckin^{2\alpha/3}(Q_1)}^{3+\frac{2\alpha}{3} + \frac{3}{\alpha}} \|f_{z_0}^\eps\|_{L^\infty(Q_1)}\\
		& \lesssim \langle v_0 \rangle^{-k-(2+\gamma)_+ +\gamma + \alpha/3} + \langle v_0 \rangle^p \langle v_0\rangle^{-k},
\end{split}
\end{equation}
where $p>0$ depends on $\alpha$, which in turn depends on $K$ and $t_0$. The implied constant in \eqref{e:Dv2f} depends on the same quantities. 
 
Translating from $f_{z_0}^\eps$ back to $f^\eps$, we clearly see that $f^\eps$ is $\Ckin^{2,2\alpha/3}$ away from $t=0$ and that
\begin{equation}\label{e:c010}
	\|\vv^{k-p} f^\eps\|_{\Ckin^{2,2\alpha/3}([t_0/2,t_0]\times \R^6)}
		\lesssim 1
\end{equation}
for some $p$ depending on $K$ and $t_0$.  We, again, stress that the implied constant in~\eqref{e:c010} depends on $t_0$ and may degenerate faster than $t_0^{-1}$ to any power since $\lambda_{t_0}$ may be exponentially small in $t_0^{-1}$.

For any $z_0\in (0,T]\times \R^6$ and $\alpha'\in (0,2\alpha/3)$, since $\Ckin^{2,\alpha'}(Q_{r_1}(z_0))$ is precompact in $\Ckin^{2,2\alpha/3}(Q_{r_1}(z_0))$, a subsequence of $f^\eps$ converges to a limit $f$ in $\Ckin^{2,\alpha'}(Q_{r_1}(z_0))$. Since $z_0 \in (0,T]\times \R^6$ is arbitrary, we have that $f \in C^{2,2\alpha/3}_{\rm kin, loc}((0,T]\times \R^6)$ and satisfies a bound such as~\eqref{e:c010}.  
Since $f^\eps \to f$ pointwise, the bound \eqref{e:kth-moment} extends to $f$. By \eqref{e:kth-moment} and the Dominated Convergence Theorem, we conclude $\bar a^{f^\eps} \to \bar a^f$, $\bar b^{f^\eps} \to \bar b^f$, and $\bar c^{f^\eps} \to \bar c^f$ as $\eps \to 0$. This establishes the existence part of the theorem.

\medskip

(ii) The proof of \eqref{e:weak_equation} relies on the following lemma:
\begin{lemma}\label{lem:epsilon_convergence} 
	Fix any compact set $K_x \subset \R^3$.  With $f^\eps$ as above, we have
	\begin{enumerate}[label=(\roman*)]
		\item $f^\e \to f$ in $L^{2,p}([0,T)\times K_x \times \R^3)$ for any $p < 7/2$,
		\item $\bar a^{f^\e} \nabla_v f^\e \to \bar a^f \nabla_v f$ weakly in $L^{2,-(1+\gamma/2)}([0,T]\times K_x \times \R^3)$, and
		\item $f^\e \bar b^{f^\e}$ converges weakly to $f \bar b^f$ in $L^{2,p}([0,T]\times K_x \times \R^3)$ for any $p < \max\{5, 15/(5+\gamma)\} - 5/2 - \gamma$.
	\end{enumerate}
\end{lemma}
We postpone the proof of this lemma momentarily and proceed with the proof of Theorem \ref{t:exist}.(ii). Recall that our test functions $\phi$ satisfy
	\[\begin{split}
	&\phi \in C([0,T]; L^{2,-7/2+\eta}(K_x\times\R^3)),
	\quad (\partial_t + v\cdot\nabla_x)\phi \in L^{2,-7/2+\eta}([0,T]\times K_x\times \R^3)\\
	&\nabla_v \phi \in L^{2,1+\gamma/2}([0,T]\times K_x \times \R^3)
	\quad \text{ and } \quad \supp(\phi) \subset [0,T)\times K_x \times \R^3,
	\end{split}\]
	for some small $\eta >0$. Note that the collision term $Q(f,f)$ in \eqref{e:landau} may be written $Q(f,f) = \nabla_v\cdot(\bar a^f\nabla_v f + \bar b^f f)$.  
Since $f^\e$ is smooth and satisfies~\eqref{e:landau}, we find
\begin{equation}\label{e:epsilon_weak_equation}
	\int_{\R^6} f_{\rm in}^\e \phi(0) \dd x \dd v
		= \int_{[0,T]\times \R^6} \left(f^\e (\partial_t + v\cdot \nabla_x) \phi - \nabla_v \phi \cdot (\bar a^{f^\e} \nabla_v f^\e) - f^\e \bar b^{f^\e} \cdot \nabla_v \phi \right) \dd x \dd v \dd t,
\end{equation}
where $\phi(0) = \phi|_{\{t=0\}}$. It is straightforward to show from the definition \eqref{e:f_in} that $f_{\rm in}^\e \to f_{\rm in}$ in $L^2_{\rm loc}$ as $\e\to 0$.  On the other hand, by \Cref{lem:epsilon_convergence}, the 
right hand side of~\eqref{e:epsilon_weak_equation} converges to
\[
	\int_{[0,T]\times \R^6} \left(f (\partial_t + v\cdot \nabla_x) \phi - \nabla_v \phi \cdot( \bar a^{f} \nabla_v f) - f \bar b^f \cdot \nabla_v \phi\right) \dd x \dd v \dd t
\]
as $\e\to0$.  Thus, we recover~\eqref{e:weak_equation} as claimed.

\medskip

(iii) For higher regularity, we return to the sequence $f_{z_0}^\eps$ and apply the argument of \cite{henderson2017smoothing}: pass the regularity provided by \eqref{e:c010} to the coefficients of the equation (if the pointwise decay in $v$ of $f^\eps$ is sufficiently strong, i.e. if $k$ is large enough compared to $p$), apply local Schauder estimates, and repeat, differentiating the equation to estimate higher derivatives in $C^{\alpha}_{\rm kin, loc}$. As $\eps\to 0$, this implies the same local regularity for $f$, with $\alpha'$ replacing $\alpha$. The number of iterations allowed is limited by the decay of $f^\eps$ in $v$ (which, by \eqref{e:kth-moment}, is determined by the decay of $f_{\rm in}$). If $f_{\rm in}$ decays faster than any polynomial, then the solution $f_{z_0}^\eps$ is $C^\infty$, which implies $f$ is $C^\infty$ as in \cite{henderson2017smoothing}. The details are omitted.
\end{proof}

We now prove the lemma.
\begin{proof}[Proof of \Cref{lem:epsilon_convergence}]

(i) We first recall that $f^\e \to f$ locally uniformly in $(0,T]\times \R^6$.  Thus, for any $\delta>0$, we have
\[\begin{split}
	\limsup_{\e\to0} &\int_0^T\int_{K_x\times \R^3} \langle v\rangle^{2p}|f^\e-f|^2 \dd x \dd v \dd t\\
		&\leq \limsup_{\e\to0} \int_0^\delta\int_{K_x \times \R^3} \langle v\rangle^{2p} |f^\e-f|^2 \dd x \dd v \dd t
			 + \limsup_{\e\to0} \int_\delta^T\int_{K_x \times B_{1/\delta}(0)}\langle v\rangle^{2p}  |f^\e-f|^2 \dd x \dd v \dd t\\
			&\qquad + \limsup_{\e\to0} \int_\delta^T\int_{K_x\times B_{1/\delta}(0)^c}\langle v\rangle^{2p}  |f^\e-f|^2 \dd x \dd v \dd t\\
		&\leq C\delta |K_x| (\limsup_{\e\to0}\|f^\e\|_{L^{\infty,5}} + \|f\|_{L^{\infty,5}})^2 + 0
			+ C(T-\delta) |K_x| (\limsup_{\e\to0}\|f^\e\|_{L^{\infty,5}} + \|f\|_{L^{\infty,5}})^2 \delta^{7-2p}.
\end{split}\]
Recall that $2p < 7$.  Taking $\delta\to0$ establishes the result.
%
%
%

\medskip

\medskip

(ii) We begin by showing that $|\bar a^{f^\e} \nabla_v f^\e|$ is bounded in $L^{2,-(1+\gamma/2)}([0,T]\times K_x\times \R^3)$.  This guarantees that $\bar a^{f^\e} \nabla_v f^\e$ has a weak subsequential limit.  First, note that
\[\begin{split}
	\|\bar a^{f^\e} \nabla_v f^\e\|_{L^{2,-(1+\gamma/2)}}^2
		&= \int_0^T \int_{K_x \times \R^3} \langle v \rangle^{-(2+\gamma)} |(\bar a^{f^\e})^{1/2} (\bar a^{f^\e})^{1/2} \nabla_v f|^2 \dd v \dd x \dd t\\
		&\lesssim \int_0^T \int_{K_x \times \R^3} \langle v \rangle^{-(2+\gamma)} \left( \langle v \rangle^{2 + \gamma} \|f^\e\|_{L^{\infty,k}}\right)|(\bar a^{f^\e})^{1/2} \nabla_v f|^2 \dd v \dd x \dd t \\
		&= \|f^\e\|_{L^{\infty,k}}\int_0^T \int_{K_x \times \R^3} \nabla_v f^\e \cdot (\bar a^{f^\e} \nabla_v f) \dd v \dd x \dd t,
\end{split}\]
since $\bar a^{f^\eps}$ is symmetric. Therefore, it is enough to show that $(\bar a^{f^\e})^{1/2} \nabla_v f^\e$ is bounded  in $L^2([0,T]\times K_x\times \R^3)$ uniformly in $\e$.  To this end, fix any non-negative $\psi\in C^\infty_c(\R^3)$ that equals $1$ on $K_x$.  Using that $f^\e$ satisfies~\eqref{e:landau} and that $\bar a^{f^\e}$ is non-negative definite, we find
\begin{equation}\label{e:c2}
\begin{split}
	0 \leq \int_0^T&\int_{K_x} \nabla_v f^\e\cdot (\bar a^{f^\e} \nabla_v f^\e) \dd x \dd v \dd t
		\leq \int_0^T\int_{\R^6} \psi \nabla_v f^\e \cdot (\bar a^{f^\e} \nabla_v f^\e) \dd x \dd v \dd t\\
		&\leq \frac{1}{2}\int_{\R^6} \psi |f_{\rm in}^\e|^2 \dd x \dd v
			+ \frac{1}{2}\int_0^T \int_{\R^6} |f^\e|^2 (v \cdot \nabla_x \psi + \frac{1}{2} \bar c^f \psi)\dd x \dd v \dd t.
\end{split}
\end{equation}
It is clear that the right hand side is bounded uniformly in $\eps$, due to the uniform $L^{\infty,k}(\R^6)$ bound that holds on $f_{\rm in}^\e$ and $f^\eps$.  

We next show that any subsequence of $\bar a^{f^\e} \nabla_v f^\e$ has a subsequence that converges weakly to $\bar a^f \nabla_v f$. It is an elementary fact that this is equivalent to the weak convergence of $\bar a^{f^\e} \nabla_v f^\e$ to $\bar a^{f} \nabla_v f$.  Fix any subsequence $\e_n \to 0$.  From the uniform bound above, we find $g \in L^{2, -(1+\gamma/2)}([0,T]\times K_x \times \R^3)^3$ and a further subsequence $\e_{n_j}\to0$ such that $\bar a^{f^{\e_{n_j}}} \nabla_v f^{\e_{n_j}}$ converges weakly to $g$ in $L^2([0,T]\times K_x\times\R^3)^3$.

On the other hand, from the Schauder estimates of the proof of part (i), and the uniform bounds on $f^\e \in L^{\infty,5}$, we find that $f^{\e_{n_j}} \in \Ckin^{2,2\alpha/3}(K)$ and $\bar a^{f^{\e_{n_j}}} \in \Ckin^{2\alpha/3}(K)$ with uniform-in-$\e$ bounds for any compact $K \subset(0,T]\times \R^6$.  Thus, up to passing to a further subsequence, $\nabla_v f^{\e_{n_j}} \to \nabla_v f$ and $\bar a^{f^{\e_{n_j}}} \to \bar a^f$ uniformly on $K$.  It follows that $g= (\bar a^f)^{1/2} \nabla_v f$ on $K$.  Since this holds for any $K$, we find $g= \bar a^f \nabla_v f$ almost everywhere in $[0,T]\times K_x\times \R^3$, which concludes the proof.

%

\medskip

(iii) We omit this proof as it is similar to and easier than the proof of (ii).
\end{proof}

\subsection{Continuation of solutions: \Cref{t:continuation}} Now, we prove our continuation criterion for the solutions of Theorem \ref{t:exist}:

\begin{proof}[Proof of Theorem \ref{t:continuation}]
%
Suppose first that $\gamma \in [-3,-2]$. If, for some $T'>0$, the quantity $\Psi(T') = \|f\|_{L^\infty_{t,x} L^p_v([0,T']\times\R^6)} + \|f\|_{L^{\infty}_{t,x} L^1_v([0,T']\times\R^6)}$ is finite, Lemma \ref{l:infty} yields
\[\|f(t,\cdot,\cdot)\|_{L^{\infty,k}(\R^6)} \leq  \|f_{\rm in}\|_{L^{\infty,k}(\R^6)} e^{C T'}, \quad 0\leq t \leq T',\]
with $C$ depending on $\Psi(T'))$. 
Since $f_{\rm in}$ is well-distributed, Lemma \ref{l:lower-bounds} with $K = \Psi(T')$ implies $f(T',\cdot,\cdot)$ is well-distributed with parameters that can only degenerate if $T'\to \infty$. Therefore, $f(T',\cdot,\cdot)$ satisfies the hypotheses of Theorem \ref{t:exist}, and we can continue the solution past $t=T'$. 

If $\gamma \in (-2,0)$, we can bound $\|f\|_{L^\infty_{t,x}L^p_v([T'/2,T']\times \R^6)}$ by applying the argument of \cite[Theorem 1.1]{cameron2017landau}. This gives an upper bound for $\|f\|_{L^\infty([T'/2,T']\times\R^6)}$ depending only on $\|f\|_{L^\infty_{t,x}L^{1,2}_v}$ and the lower ellipticity bound for $\bar a^f$ on $[T'/2,T']$ given by Lemma \ref{l:lower-bounds} (which can be bounded in terms of $\Psi(T')$ and the initial data). 
This gives a bound for $f$ in $L^\infty_x L^p_v$ for $t\in [T'/2,T']$, which, combined with the reasoning of the previous paragraph, lets us
continue the solution past $t=T'$.
\end{proof}

\section{Pointwise matching of initial data}\label{s:initial}

In this section we show that, under the additional assumption that $f_{\rm in}$ is continuous, we have $f_{\rm in}(x,v) = \lim_{t\to 0+} f(t,x,v)$. The proof uses a simple barrier argument.

\begin{proposition}\label{p:initial}
	Let $f$ be the solution to \eqref{e:landau} constructed in \Cref{t:exist}. 
If $f_{\rm in} = f(0,\cdot,\cdot)\in L^{\infty,k}(\R^6)$ is continuous, then $f(t,x,v) \to f_{\rm in}(x,v)$ as $t\to 0+$, uniformly on compact sets of $\R^6$.  If $f_{\rm in}$ is uniformly continuous in $x$, then the convergence as $t\to 0+$ is uniform.
\end{proposition}

\begin{proof}
We work with the approximating solutions $f^\e$ from \Cref{t:exist}, obtaining a uniform bound in $\e$.  Thus, we obtain the result in the limit $\e\to0$.  Importantly, these $f^\e$ are smooth on $[0,T]\times\R^6$, so we may use the classical comparison principle.  In an abuse of notation, we denote $f^\e$ simply by $f$ for the remainder of the proof.

We show that continuity of $f_{\rm in}$ at a fixed $(x_0,v_0)\in \R^6$ implies $f(t,x_0,v_0) \to f_{\rm in}(x_0,v_0)$ as $t\to 0$. 
Fix $\eta>0$.  If $|v_0|$ is sufficiently large, then the finiteness of $\|f\|_{L^{\infty,k}([0,T]\times\R^6)}$ guarantees that $f(t,x,v_0) < \eta$ for all $x$ and for all $t$ sufficiently small.  Hence, we need only consider $v_0 \in B_{R_\eta}(0)$ for some large $R_\eta>0$ depending only on $\eta$ and $\|f\|_{L^{\infty,k}([0,T]\times \R^6)}$.

Let $\delta>0$ be such that 
\[
	|f_{\rm in}(x,v) - f_{\rm in}(x_0,v_0)| < \eta,
		\quad \mbox{ if } |x-x_0|^2 + |v-v_0|^2 \leq \delta^2.
\]
 With $M,\beta,\rho>0$ to be determined, define
\[
	\overline h(t,x,v)
		= e^{\beta t} \left[ M( |x-x_0-vt|^2 + |v-v_0|^2) + \eta + f_{\rm in}(x_0,v_0) + \rho t\right].
\]
Recall the linear operator $Lg = \partial_t g + v\cdot \nabla_x g - \tr(\bar a^f D_v^2 g) - \bar c^f g$ from the proof of Lemma \ref{l:infty}. We have $D_v^2 \overline h = 2 M e^{\beta t}(1+t^2) \mbox{Id}$, and
\begin{align*}
	L \overline h
		&= \beta \overline h + \rho e^{\beta t}
			 -2 Me^{\beta t}(1+t^2) \tr(\bar a^f ) - \bar c^f \overline h.
\end{align*}
Since $\bar c^f \lesssim \|f\|_{L^\infty([0,T],L^{\infty,k}(\R^6))}$ and $\bar a^f \lesssim \vv^{(\gamma+2)_+}
\|f\|_{L^\infty([0,T],L^{\infty,k}(\R^6))}$ (by Lemma \ref{l:abc}), one has $L\overline h \geq 0$ in $[0,1]\times \R^6$, provided $\beta$ and $\rho$
are large enough. By our choice of $\delta$, we have
\[f_{\rm in}(x,v) \leq f_{\rm in}(x_0,v_0) + \eta \leq \overline h(0,x,v), \quad \mbox{ for } |x-x_0|^2 + |v-v_0|^2 \leq \delta^2.\]
Next, choose $M>0$ large enough so that, for $t \leq \delta /(4(R_\eta +\delta))$ and $|x-x_0|^2 + |v-v_0|^2 = \delta^2$, there holds
\[
\begin{split}
\overline h(t,x,v) &\geq M(|x-x_0|^2+t^2|v|^2 - 2t (x-x_0)\cdot v + |v-v_0|^2)\\
 &\geq M(\delta^2 - 2t(x-x_0)\cdot v)\\
  &\geq M\delta^2/2 \geq \|f\|_{L^\infty([0,T]\times \R^6)}.
\end{split}
\]
We may now apply the maximum principle to obtain $f(t,x,v) \leq \overline h(t,x,v)$ in $[0, \delta/(4(R_\eta +\delta))]\times B_\delta(x_0,v_0)$.

Let $\delta' =  \sqrt{\eta/M}$ and 
\[
	t_0=\min\left\{\frac{1}{\beta}\log(1+\eta),\ \frac{\delta'}{(|v_0|+\delta')}, \ \frac{\eta}{\rho}\right\}.
\]
If $(t,x,v) \in [0,t_0) \times B_{\delta'}(x_0,v_0),$ we have
\[
	|x-x_0-tv|^2 + |v-v_0|^2
		\leq 2|x-x_0|^2  + |v-v_0|^2 + 2t^2|v|^2
		< 4(\delta')^2,
\]
and, thus, 
\[
\begin{split}
	 f(t,x,v)
	 	&\leq \overline h(t,x,v)
	 	< e^{\beta t} \left[ 4 M (\delta')^2 + \eta + f_{\rm in}(x_0,v_0) + \rho t\right]
	 	< (1+\eta)\left[ 4 \eta + \eta + f_{\rm in}(x_0,v_0) + \eta\right].
\end{split}
\]
Since $\eta$ was arbitrary, we conclude $\limsup_{t\to 0} f(t,x_0,v_0) \leq f_{\rm in}(x_0,v_0)$.

For the reverse inequality, define
\[\underline h(t,x,v) =  e^{-\beta t} \left[f_{\rm in} - M(|x-x_0-vt|^2 - |v-v_0|^2) - \eta -\rho t\right].\]
By a similar calculation, we have $L\underline h \leq 0$ for $t\leq 1$, and $f_{\rm in}(t,x,v) \geq \underline h(t,x,v)$ on the parabolic boundary of $[0,\delta/(2(|v_0|+\delta))]\times B_\delta(x_0,v_0)$, if $M,\beta,\rho$ are chosen large enough. The rest of the proof follows similarly to establish that
$f_{\rm in}(x_0,v_0) = \lim_{t\to 0+} f(t,x_0,v_0)$. 

Clearly, $\delta$ (and therefore $M$, $\delta'$, and $t_0$) can be chosen uniformly on any compact set of $\R^6$. If $\delta$ is independent of $x_0$, then $t_0$ can be chosen depending only on $\eta$ and $\|f\|_{L^{\infty,k}([0,T]\times\R^6)}$ (recall that $R_\eta$ depends only on $\eta$ and $\|f\|_{L^{\infty,k}([0,T]\times \R^6)}$), which yields the uniform convergence.  Thus, the proof is complete.
\end{proof}

\section{Propagation of H\"older regularity and higher regularity estimates}\label{s:holder}

In this section, we prove that if $f_{\rm in}(x,v)$ is H\"older continuous, with H\"older norm decaying appropriately for large velocities, then our solution $f$ is H\"older continuous in $(t,x,v)$ up to some short time $T_H$. This is a necessary ingredient of our proof of uniqueness (see Section \ref{s:unique}).



In this section, we are not interested in the dependence of constants on $\|f\|_{L^{\infty,k}}$.  Thus, we allow the implied constant in any ``$\lesssim$'' to depend on $\|f\|_{L^{\infty,k}([0,T]\times\R^6)}$.

Following the proof outline given in the introduction, the first step is to revisit the Schauder estimates from the proof of Theorem \ref{t:exist}.(i) under stronger assumptions on $f_{\rm in}$ that ensure good lower bounds for $f$ as $t\to 0$. The non-scale-invariance of this estimate reflects the dependence of the coefficients in \eqref{e:landau_nd} on $f$.

\begin{lemma}\label{l:weighted-Schauder}
Let $f \in L^{\infty,k}([0,T]\times \R^6)$ be a solution to \eqref{e:landau} with $f_{\rm in}$ satisfying the hypotheses of Theorem \ref{t:exist}. Assume in addition that for all $x$, there is some $|v_m|<R$ with $f_{\rm in}(\cdot,\cdot)\geq \delta \1_{B_r(x,v_m)}$. For any $t_0\in (0,T]$, we have
\[
	\|\vv^{m-q(\gamma,\alpha,k,m)}D_v^2 f\|_{L^\infty([t_0/2,t_0]\times \R^6)}
		\lesssim  t_0^{-1+\alpha^2/(6-\alpha)}(1+ \|\vv^m f\|_{\Ckin^\alpha([t_0/2,t_0]\times \R^6)})^{p(\alpha)},
\]
for any $m\in(\max\{3,  5 + \gamma + \alpha/3\}, k]$ such that the right-hand side is finite, with
\[
\begin{split}
	p(\alpha)
		&= 3+2\alpha/3+3/\alpha,\\
	q(\gamma,\alpha, k, m)
		&= (2+\gamma)_+ - \gamma\\
		&\quad + \left(1-\frac{\alpha^2}{6-\alpha}\right) \max\left\{-(2+\gamma)_+ + \gamma - (k-m)/3, (2 + \alpha/3)p(\alpha) - k + m\right\}
\end{split}
\] 
and the implied constant depending only on $m$, $\gamma$, $\alpha$, $T$, $\delta$, $r$, $R$, $k$, and $\|f\|_{L^{\infty,k}([t_0/2,t_0]\times\R^6)}$.
\end{lemma}
\begin{proof}
Let $z_0 \in (0,T]\times \R^6$ be fixed, and let $f_{z_0}(z) = f(\mathcal S_{z_0}(\delta_{r_1} z))$, with $r_1$ defined by~\eqref{e:r_1} and $\mathcal S_{z_0}$ defined by~\eqref{e:COV} as in the proof of Theorem \ref{t:exist}. The function $f_{z_0}$ satisfies \eqref{e:isotropic-nondivergence} in $Q_1$. By \Cref{l:lower-bounds}, \Cref{l:transformation_L_infinity}, and our extra assumption on $f_{\rm in}$, the diffusion matrix $\overline A(z)$, given by~\eqref{e:COV_coefficients}, satisfies upper and lower ellipticity estimates that are independent of $z_0$. In other words, $\lambda_{t_0}$ in \eqref{e:lambdaLambda} depends only on the initial data and $T$, not on $t_0$.

Our goal is to apply the Schauder estimate \cite[Theorem 2.12(a)]{henderson2017smoothing}.   
Using the bounds in \Cref{l:Holder_bar_A}, along with the fact that $[f_{z_0}]_{\Ckin^{2\alpha/3}(Q_1)}\leq [f_{z_0}]_{\Ckin^{\alpha}(Q_1)}$, we find
\begin{equation}\label{e:schauder-application}
\begin{split}
[D_v^2 & f_{z_0}]_{\Ckin^{2\alpha/3}(Q_{1/2})}
	\lesssim [\overline C f_{z_0}]_{\Ckin^{2\alpha/3}(Q_1)} + \|\overline A\|^{p(\alpha)}_{\Ckin^{2\alpha/3}(Q_1)}\|f_{z_0}\|_{L^\infty(Q_1)}\\
	&\lesssim [\overline C]_{\Ckin^{2\alpha/3}(Q_1)} \|f_{z_0}\|_{L^\infty(Q_1)}
		+ \|\overline C\|_{L^\infty(Q_1)} \|f_{z_0}\|_{\Ckin^{2\alpha/3}(Q_1)}\\
		&\qquad+ \vvO^{(2 + \alpha/3)p(\alpha)}\|\vv^m f\|^{p(\alpha)}_{\Ckin^\alpha} \vvO^{-k}\\
	&\lesssim \left(\vvO^{-(2+\gamma)_+ + \gamma + \alpha/3}  \vvO^{-k}
		+ \vvO^{-(2+\gamma)_+ + \gamma - k/3 - 2m/3}\right) \|\vv^m f\|_{\Ckin^\alpha}\\
		&\qquad+ \vvO^{(2 + \alpha/3)p(\alpha) - k}\|\vv^m f\|^{p(\alpha)}_{\Ckin^\alpha}\\
	&\lesssim \vvO^{-m + \tilde q(\alpha,\gamma,k,m)} \left(1 + \|\vv^m f\|_{\Ckin^\alpha}\right)^{p(\alpha)},
\end{split}
\end{equation}
where, as above, we use the shorthand $\Ckin^\alpha = \Ckin^\alpha([t_0/2,t_0]\times \R^6)$, and we have defined
\[
	\tilde q(\alpha, \gamma, k, m) = \max\left\{-(2+\gamma)_+ + \gamma - (k-m)/3, (2 + \alpha/3)p(\alpha) - (k - m)\right\}.
\]
The reduction to the maximum of these two values is due to the fact that $-(2+\gamma)_+ + \gamma + \alpha/3 - (k-m)\leq (2+\alpha/3)p(\alpha) - (k-m)$.  In the third line of \eqref{e:schauder-application}, we also used an interpolation between $\Ckin^\alpha$ and $L^{\infty,k}$ (see Lemma \ref{lem:holder_decay}) to write
\begin{equation}\label{e:interp1}
	[f_{z_0}]_{\Ckin^{2\alpha/3}(Q_1)}
		\lesssim \|f_{z_0}\|_{L^\infty(Q_1)}^{1/3} [f_{z_0}]_{\Ckin^\alpha(Q_1)}^{2/3}
		\lesssim \min\{1, t_0^{\alpha/3}\} \vvO^{-k/3 - 2m/3} [\vv^m f]_{\Ckin^\alpha}^{2/3},
\end{equation}
and used $t_0 \lesssim 1$.

Using \eqref{e:interp1} again, and an interpolation between $C_v^\alpha$ and $C_v^{2,\alpha}$ (see \Cref{lem:holder_interpolation}), we find
\[\begin{split}
	\|D^2_v &f_{z_0}\|_{L^\infty(Q_{1/4})}
		\lesssim [f_{z_0}]_{\Ckin^\alpha(Q_{1/2})} + [f_{z_0}]_{C_v^{2\alpha/3}(Q_{1/2})}^{\frac{2\alpha/3}{2 + (2\alpha/3)-\alpha}} [D^2_v f_{z_0}]_{C_v^\alpha(Q_{1/2})}^{1-\frac{2\alpha/3}{2 + (2\alpha/3)-\alpha}}\\
		&\lesssim \left(\min\{1, t_0^{\alpha/2}\} \vvO^{-m} + \min\{1, t_0^{\frac{\alpha}{2}\frac{2\alpha}{6 - \alpha}}\}\vvO^{-m\frac{2\alpha}{6 - \alpha}} \vvO^{\left(1-\frac{2\alpha}{6-\alpha}\right)(-m + \tilde q(\alpha,\gamma,k,m))}\right)\\
			&\qquad (1+ \|\vv^m f\|_{\Ckin^\alpha})^{p(\alpha)}\\
		&\lesssim \min\{1, t_0^\frac{\alpha^2}{6-\alpha}\} \vvO^{-m + \left(1-\frac{2\alpha}{6-\alpha}\right) \tilde q(\alpha,\gamma,k,m)} (1+\|\vv^m f\|_{\Ckin^\alpha})^{p(\alpha)}.
\end{split}\]
Undoing the change of variables, we get that
\[
	|D^2_v f(z_0)|
		\lesssim |r_1^{-2} S^{-2} D_v^2 f_{z_0}(0)|
		\lesssim \left(1 + t_0^{-1 + \frac{\alpha^2}{6-\alpha}}\right) \vvO^{((2+\gamma)_+ - \gamma) - m + \left(1-\frac{\alpha^2}{6-\alpha}\right)\tilde q(\alpha,\gamma,k,m)} \|\vv^m f\|_{\Ckin^\alpha}.
\]
Since $q(\alpha,\gamma,k,m) = ((2+\gamma)_+ - \gamma) + \left(1-\frac{\alpha^2}{6-\alpha}\right) \tilde q(\alpha,\gamma,k,m)$ and since $z_0$ was arbitrary, the claim is proved.
\end{proof}

%

The purpose of the next lemma is passing from H\"older regularity in $(x,v)$ to H\"older regularity in $(t,x,v)$.  This lemma is proven in Appendix \ref{s:reg}, in a more general form. Since the reverse implication is immediate, we see that bounding our solution $f$ in $L^\infty_t C^\alpha_{{\rm kin},x,v}$ is equivalent to bounding $f$ in $\Ckin^\alpha$, up to velocity weights.

\begin{lemma}\label{l:holder-vx-to-t}
For any locally H\"older continuous solution to \eqref{e:landau} with $f\in L^{\infty,k}([0,T]\times \R^6)$ and $k$ as in Theorem \ref{t:exist}, there holds
\begin{equation*}
	\|f\|_{\Ckin^{\alpha}(Q_1(z_0)\cap [0,T]\times \R^6)}
		\lesssim \langle v_0\rangle^{\alpha (1+\gamma/2)_+} \left(\|f\|_{L^\infty([0,T]\times \R^6)} + \|f\|_{L^\infty_t C^\alpha_{{\rm kin},x,v}(Q_2(z_0) \cap [0,T]\times \R^6)}\right),
\end{equation*}
for any $z_0 \in [0,T]\times \R^6$, 
where, for any $A\subset \R^6$,
\[
	[f]_{C^\alpha_{{\rm kin},x,v}(A)}
		= \sup_{(x_1,v_1),(x_2,v_2)\in A} \frac{|f(x_1,v_1) - f(x_2,v_2)|}{(|x_1-x_2|^{1/3} + |v_1+v_2|)^\alpha}.
\]
The implied constant depends only on $\|f\|_{L^{\infty,k}}$.
\end{lemma}

By combining \Cref{l:holder-vx-to-t} and \Cref{l:weighted-Schauder}, we deduce the following, whose proof is omitted.

%
%
%
%
%
\begin{lemma}\label{l:D2f-Linfty}
Let the assumptions of \Cref{l:weighted-Schauder} hold and additionally assume that $m > \max\{3, 5 + \gamma + \alpha/3\} + \alpha(1+\gamma/2)_+$.  Let
\[
	 q'(\alpha,\gamma,k,m) = q(\alpha, \gamma, k, m - \alpha(1+\gamma/2)_+) + \alpha(1+\gamma/2)_+.
\]
Then
\[
	\|\vv^{m-  q'(\gamma,\alpha,k,m)} D^2_v f\|_{L^\infty}
		\lesssim t_0^{-1 + \alpha^2/(6-\alpha)} ( 1 + \|\vv^m f\|_{L^\infty_tC^\alpha_{{\rm kin},x,v}([t_0/2,t_0]\times \R^6)})^{p(\alpha)}
\]
\end{lemma}

Now we are ready to show that H\"older continuity at $t=0$ implies H\"older continuity for positive time. The proof requires us to work with the Euclidean H\"older norm $\|\cdot\|_{C^\alpha(\R^6)}$.  This norm is always only in $x,v$ variables.

\begin{proposition}\label{prop:holder_propagation}
	Let $f$ be the solution constructed in Theorem \ref{t:exist}.  Suppose that $\langle v\rangle^{m} f_{\rm in} \in C^\alpha(\R^6)$ and $f_{\rm in} \in L^{\infty,k}(\R^6)$ and that $\alpha$, $m$, and $k$ satisfy
	\[
		m > \max\{3,5 + \gamma + \alpha/3\} + \alpha(1+\gamma/2)_+
			\quad \text{ and } \quad
		q'(\alpha,\gamma,k,m)
			\leq (2+\gamma)_+.
	\]
	Then there exists $T_H \in (0,T]$ such that
	\[
		\|\langle v\rangle^{m} f\|_{L^\infty_t C^\alpha([0,T_H]\times \R^6)}
			\lesssim \|\langle v\rangle^{m} f_{\rm in}\|_{C^\alpha(\R^6)}.
	\]
	The implied constant above and $T_H$ depend only on $m$, $k$, $\alpha$, $\gamma$, $\|f\|_{L^{\infty, k}([0,T]\times \R^6)}$, $\delta$, $r$, and $R$.
\end{proposition}

\begin{proof}

Let $f^\eps$ be the regularizing approximation from the proof of Theorem \ref{t:exist}. We show that the conclusion of the proposition holds for $f^\eps$, with $T_H$ independent of $\eps$, so the same conclusion for $f$ follows. The smoothness and decay of $f^\eps$ is used to obtain a first touching point with a supersolution, and to ensure the right-hand side of Lemma \ref{l:weighted-Schauder} is 
finite, but none of the estimates depend quantitatively on $\eps$. To keep the notation clean, we denote $f = f^\eps$ for this proof.

\medskip

\noindent {\it Step 1: Defining $g$ and deriving its equation.} As discussed in the introduction, for $(t,x,v,\chi,\nu) \in \R_+ \times \R^6 \times B_1(0)^2$ and $m > 0$, we define
\begin{equation}
\begin{split}
&\tau f(t,x,v,\chi,\nu) := f(t,x+\chi, v+\nu),\\
&\delta f (t,x,v,\chi,\nu) = \tau f(t,x,v,\chi,\nu) - f(t,x,v),\\
&g(t,x,v,\chi,\nu) = \frac{|\delta f (t,x,v,\chi,\nu)|^2}{(|\chi|^{2} + |\nu|^2)^{\alpha}} \vv^{2m}.
\end{split}
\end{equation}
Note that $g$ encodes the Euclidean H\"older norm of $f$, as stated in the following elementary lemma.

\begin{lemma}\label{lem:holder_char}
	Let $m\geq 0$.  Fix any $f : \R^6 \to \R$ and let $g: \R^6 \times B_1(0)^2 \to \R$ be defined by $g(x,v,\chi,\nu) = |\delta f(x,v,\chi,\nu)|^2 \langle v \rangle^{2m} / (|\chi|^2+|\nu|^2)^{\alpha}$.  There holds
	\[
		\|g\|_{L^\infty(\R^6\times B_1(0)^2)} + \|\vv^m f\|_{L^\infty(\R^6)}^2 \approx \|\langle v \rangle^m f\|_{C^\alpha(\R^6)}^2 \approx \sup_{x,v} \vv^{2m} \|f\|_{C^\alpha(B_1(x,v))}^2,
	\]
	where the implied constants depend only on $m$ and $\alpha$.
\end{lemma}
Hence, deriving a time-dependent upper bound on $g$ suffices to prove the proposition. 
We obtain a bound on $g$ by showing that it satisfies an equation where all terms are either bounded nicely or respect a maximum principle.  By a short computation, $g$ satisfies
\begin{equation}\label{e:Holder}
\begin{split}
& \partial_t g + v\cdot \nabla_x g + \nu \cdot \nabla_{\chi} g
+\frac{2\alpha \nu \cdot \chi}{|\chi|^2 + |\nu|^2} g\\
	&\qquad\qquad= 2\frac{\left( \tr(\bar{a}^{\delta f} D^2_v \tau f + \bar a^f D^2_v \delta f)
		+ \bar c^{\delta f} \tau f + \bar c^f \delta f \right) \delta f}{(|\chi|^2 + |\nu|^2)^{\alpha}} \vv^m.
\end{split}
\end{equation}
The first three terms on the left enjoy a maximum principle.  The last term on the left is clearly bounded by $g$.  The terms on the right can be shown to be bounded, which is the main thrust of the argument below.


Before continuing with the proof, we discuss why the Euclidean H\"older norm occurs naturally here.  In the definition of $g$, one might expect to see a denominator of $(|\chi|^2 + |\nu|^6)^{\alpha/3}$ or another term with this balance of powers that respects the kinetic scaling that is natural to the equation.  However, this would replace the last term on the left hand side of~\eqref{e:Holder} by
\[
	\frac{2\alpha}{3} \frac{\nu \cdot \chi}{|\chi|^2 + |\nu|^6} g
\]
which is not bounded. This forces the choice of $(|\chi|^2 + |\nu|^2)^\alpha$ for the denominator of $g$, which is why this proposition is stated in terms of the Euclidean H\"older norm.


To find a time-dependent upper bound of $g$ that remains finite for some time interval, we use~\eqref{e:Holder} and construct a super-solution of $g$.

\medskip

\noindent {\it Step 2: A super-solution and the maximum principle argument.}   
With $N > 0$ to be chosen later, define $\overline G$ to be the unique solution to
\begin{equation}\label{e:g_supersoln}
\begin{cases}
	\frac{d}{dt} \overline G(t) = N t^{-1 + \frac{\alpha^2}{6-\alpha}}\left(1 + \overline G(t)\right)^{\frac{p(\alpha)+1}{2}},\\
	\overline G(0) = 1 + \|g(0,\cdot)\|_{L^\infty (\R^6 \times B_1(0)^2)} + N \|f\|^2_{L^{\infty,m}([0,T]\times \R^6)}.
\end{cases}
\end{equation}
This solution $\overline G$ exists on a maximal time interval $[0,T_G)$ with $T_G$ depending on $N$, $\beta$, $\alpha$, $\|g(0,\cdot)\|_{L^\infty (\R^6 \times B_1(0)^2)}$, and $\|f\|_{L^\infty([0,T]\times \R^6)}$.  Our goal is to show that $g(t,x,v,\chi,\nu) < \overline G(t)$ for all $t\in [0,\min\{T,T_G\})$. By \Cref{lem:holder_char}, this implies the existence of $T_H$ as in the statement of the proposition. Let $t_0$ be the first time that $\|g\|_{L^\infty([0,t_0]\times \R^6 \times B_1(0)^2)} = \overline G(t_0)$. It is clear, by construction, that $t_0 >0$. We seek a contradiction at $t=t_0$.

First, we claim that we may assume there exists $(x_0, v_0, \chi_0, \nu_0) \in \R^6 \times \overline{B_1}(0)^2$ such that $g(t_0,x_0, v_0, \chi_0, \nu_0) = \overline G(t_0)$.  If there is no such point, then fix any sequence $z_n \in \R^6\times B_1(0)^2$ such that $g(t_0,z_n) \to \overline G(t_0)$. Recall that $f$ (which is actually the regularization $f^\eps$) is $C^\infty$ and satisfies pointwise Gaussian decay in $v$. Because of this Gaussian decay, we can take $z_n \in  \R^3 \times B_R(0)\times B_1(0)^2$ for some $R>0$.  Since $g$ does not decay as $|x|\to \infty$, we need to re-center as follows: let $g_n(t,x,v,\chi,\nu) = g(t, x + x_n, v, \chi, \nu)$.  Up to passing to a subsequence, it is clear that there exists $\bar g$ and $(v_0, \chi_0, \nu_0) \in \overline{B_R}(0) \times \overline{B_1}(0)^2$ such that $g_n \to \bar g$ locally uniformly and $\bar g(t_0, 0, v_0,\chi_0, \nu_0) = \overline G(t_0)$.  Further, $\|\bar g\|_{L^\infty([0,t]\times \R^6 \times B_1(0)^2)} < \overline G(t)$ for all $t< t_0$. The smoothness of $g$ implies that, again up to passing to a subsequence, $g_n \to \bar g$ in $C^{\ell}_{loc}$ for any $\ell\in\N$, so that $\bar g$ satisfies the same equation as $g$, i.e. \eqref{e:Holder}. The proof may then proceed using $\bar g$ in the place of $g$.  Hence, we may assume the existence of $(x_0,v_0,\chi_0,\nu_0)$.

Next, we notice that $(\chi_0,\nu_0) \in B_1(0)^2$.  Indeed, if $\chi_0$ or $\nu_0$ were in $\partial B_1(0)$, then
\[\begin{split}
	g(t_0,x_0, v_0, \chi_0, \nu_0)
		&\leq \frac{|\delta f(t_0, x_0, v_0, \chi_0, \nu_0)|^2}{1} \langle v_0\rangle^{2m}\\
		&\lesssim f(t_0, x_0 + \chi_0, v_0 + \nu_0)^2 \langle v_0 \rangle^{2m} + f(t_0, x_0, v_0)^2\langle v_0\rangle^{2m}
		\lesssim \|f\|^2_{L^{\infty,m}}.
\end{split}\]
Then, choosing $N$ sufficiently large, we find $g(t_0, x_0, v_0, \chi_0,\nu_0) \leq N \|f\|^2_{L^{\infty,m}}$.  On the other hand, by~\eqref{e:g_supersoln}, $\overline G$ is increasing.  Hence, we have $\overline G(t_0) \geq \overline G(0) > N\|f\|^2_{L^{\infty,m}}$, which contradicts the fact that $\overline G(t_0) = g(t_0, x_0, v_0, \chi_0, \nu_0)$.  Thus, we conclude $(\chi_0, \nu_0) \in B_1(0)^2$.

In order to conclude the proof, we establish that, at $(t_0, x_0, v_0, \chi_0, \nu_0)$,
\begin{equation}\label{e:g_comparison}
	\partial_t g
		< \frac{N}{t^{1 - \frac{\alpha^2}{6-\alpha}}}\left(1 + g\right)^{\frac{p(\alpha)+1}{2}}
\end{equation}
as long as $N$ is chosen sufficiently large, depending on $\alpha$, $m$, $k$, and the constant in \Cref{l:weighted-Schauder}. Since, at this location, there holds $g= \overline G$ and $\partial_t g  \geq (d/dt) \overline G$, this yields a contradiction in view of~\eqref{e:g_supersoln}.

\medskip

\noindent {\it Step 3: Re-writing~\eqref{e:Holder} at $(t_0,x_0,v_0,\chi_0,\nu_0)$.} 
In order to prove that~\eqref{e:g_comparison} holds, we return to~\eqref{e:Holder} and examine the right hand side.  Notice that 
\begin{equation}\label{e:prop_c1}
\begin{split}
	2&\frac{\tr(\bar{a}^{\delta f} D^2_v \tau f + \bar a^f D^2_v \delta f)}{(|\chi|^2 + |\nu|^2)^{\alpha}} \delta f \vv^{2m} \\
		&\qquad = \tr(\bar a^f D^2_v g) +  2\frac{\tr(\bar a^{\delta f} D^2_v(\tau f)) }{(|\chi|^2 + |\nu|^2)^{\alpha}}\delta f\vv^{2m}
				- 2 \nabla_v(\delta f)\cdot( \bar a^f \nabla_v (\delta f)) \frac{\vv^{2m}}{(|\chi|^2 + |\nu|^2)^{\alpha}}\\
		&\quad\qquad - \frac{2m \vv^{2m-2}}{(|\chi|^2 + |\nu|^2)^{\alpha}}\left( 4 v \cdot(\bar a^f \nabla_v(\delta f)) \delta f + \tr(\bar a^f) |\delta f|^2 + \frac{(2m-2) v\cdot (\bar a^f v)}{\vv^2} |\delta f|^2 \right).
%
\end{split}
\end{equation}
On the other hand $(x_0,v_0,\chi_0,\nu_0)$ is the location of a global maximum of $g$.  Hence, at $(x_0,v_0,\chi_0,\nu_0)$, we obtain the identity
\[
	0 = \nabla_v g
		= \langle v_0\rangle^{2m} \frac{2 \delta f \nabla_v(\delta f)}{(|\chi_0|^2 + |\nu_0|^2)^{\alpha}}
			+ 2m v_0 \langle v_0\rangle^{2m-2} \frac{|\delta f|^2}{(|\chi_0|^2 + |\nu_0|^2)^{\alpha}},
\]
or $\nabla_v(\delta f) = -m \langle v_0\rangle^{-2} \delta f v_0$. Then ~\eqref{e:prop_c1} becomes, at $(x_0,v_0,\chi_0,\nu_0)$,
\[
\begin{split}
	&2\frac{\tr(\bar{a}^{\delta f} D^2_v \tau f + \bar a^f D^2_v \delta f)}{(|\chi_0|^2 + |\nu_0|^2)^{\alpha}}\delta f \langle v_0\rangle^{2m}\\
		&\qquad = \tr(\bar a^f D^2_v g) +  2\frac{\tr(\bar a^{\delta f} D^2_v(\tau f))}{(|\chi_0|^2 + |\nu_0|^2)^{\alpha}} \delta f\langle v_0\rangle^{2m}
			+ \frac{2m g}{ \langle v_0\rangle^4} \left( \left(m +2\right) v_0 \cdot(\bar a^f v_0) - \langle v_0 \rangle^2 \tr(\bar a^f)\right).
\end{split}
\]
Using again that $(x_0,v_0,\chi_0,\nu_0)$ is the location of a maximum, we have $\bar a^f D_v^2 g \leq 0$ and $\nabla_x g = \nabla_\chi g = 0$.  Therefore, from \eqref{e:Holder} we obtain, at $(x_0,v_0,\chi_0,\nu_0)$,
\begin{equation}\label{e:holder_c3}
\begin{split}
	\partial_t g
		&\leq -2\alpha \frac{\nu_0 \cdot \chi_0}{|\chi_0|^2 + |\nu_0|^2} g
			+  2\frac{\tr(\bar a^{\delta f} D^2_v(\tau f) )}{(|\chi_0|^2 + |\nu_0|^2)^{\alpha}}\delta f \langle v_0\rangle^{2m}
			+ \frac{2m g}{ \langle v_0\rangle^4} \left( \left(m +2\right) v_0 \cdot (\bar a^f v_0) - \langle v_0 \rangle^2 \tr(\bar a^f)\right)\\
		&\quad  + 2\frac{ \bar c^{\delta f} \tau f \delta f}{(|\chi_0|^2 + |\nu_0|^2)^{\alpha}} \langle v_0\rangle^{2m} + 2\bar c^f  g\\
		&=: J_1 + J_2 + J_3 + J_4 + J_5.
\end{split}
\end{equation}
It is clear that $J_1, J_3, J_5\lesssim g$ (see Lemma \ref{l:abc} for $J_3$ and $J_5$). We now bound the remaining two terms.

\medskip

\noindent {\it Step 4: Bounding $J_2$ and $J_4$.}  
Re-writing $J_4$ and taking the absolute value, we find
\[
	|J_4|
		\leq \frac{2 |\bar c^{\delta f}| |\tau f| \sqrt g}{(|\chi_0|^2 + |\nu_0|^2)^{\alpha/2}} \langle v_0 \rangle^{m}.
\]
It is clear that we must bound the H\"older modulus of $\bar c^{\delta f}$. If $\gamma \in (-3,0)$, then using Lemma \ref{lem:holder_char}, we have
\begin{equation*}
\begin{split}
 \frac{\bar c^{\delta f}}{(|\chi_0|^2+|\nu_0|^2)^{\alpha/2}}
 	&= c_\gamma \int_{\R^3} |w|^\gamma \frac{\delta f (t_0, x_0, v_0-w, \chi_0,\nu_0)}{(|\chi_0|^2+|\nu_0|^2)^{\alpha/2}} \dd w\\
&\leq c_\gamma \int_{\R^3} \frac{|w|^\gamma}{ \langle v_0-w \rangle^{m}} \|g(t_0,\cdot)\|_{L^\infty(\R^6\times B_1(0)^2)}^{\frac 1 2} \dd w
\leq C \|g(t_0,\cdot)\|_{L^\infty(\R^6\times B_1(0)^2)}^{\frac 1 2},
\end{split}
\end{equation*}
because $ m  > 3 + \gamma$. On the other hand, if $\gamma = -3$, we have (up to a constant) $\bar c^{\delta f} = \delta f$, and Lemma \ref{lem:holder_char} directly implies the same upper bound for $\bar c^{\delta f} (|\chi_0|^2 + |\nu_0|^2)^{-\alpha/2}$. In addition, it is clear that
$|\tau f |\vv^{m} \lesssim \| f \|_{L^{\infty,k}}$ because $k \geq m$.  By construction, we have $\|g(t_0,\cdot)\|_{L^\infty(\R^6\times B_1(0)^2)} = g(t_0,x_0,v_0,\chi_0,\nu_0)$.  Thus, at $(t_0,x_0,v_0,\chi_0,\nu_0)$, we have
\begin{equation*}
 J_4 \leq C g.
\end{equation*}

%
%
%
For $J_2$, we begin with a similar approach; observe that
\begin{equation}
 \begin{split}
	\frac{\bar{a}^{\delta f}}{(|\chi_0|^2+|\nu_0|^2)^{\alpha/2}}
		&\leq a_\gamma \int_{\R^3} |w|^{2+\gamma} \frac{\delta f (t_0, x_0, v_0-w, \chi_0,\nu_0)}{(|\chi_0|^2+|\nu_0|^2)^{\alpha/2}} \dd w\\
		&\leq a_\gamma \int_{\R^3} \frac{|w|^{2+\gamma}}{\langle v_0-w \rangle^{m}} \| g(t_0,\cdot) \|_{L^\infty(\R^6\times B_1(0)^2)}^{\frac 1 2} \dd w\\
		&\leq C   \| g(t_0,\cdot) \|_{L^\infty(\R^6\times B_1(0)^2)}^{1 / 2} \langle v_0\rangle^{(2+\gamma)_+}.
 \end{split}
\end{equation}
This holds since $m > 5 + \gamma$ by assumption.
Thus
\begin{equation}\label{e:J2bound}
	J_2
		\lesssim |D^2_v (\tau f)| \langle v_0 \rangle^{m + (2+\gamma)_+} \| g(t_0,\cdot) \|_{L^\infty(\R^6\times B_1(0)^2)}^{1/2}.
\end{equation}
To close the estimate, we require an upper bound on
$\| \vv^{m + (2+\gamma)_+} D^2_v f \|_{L^\infty}$, which is provided by \Cref{l:weighted-Schauder}; however, the results in \Cref{l:weighted-Schauder} require working with the kinetic H\"older norms.  In this case, it suffices to notice that
\[
	\|\vv^m f\|_{C_{{\rm kin},x,v}^\alpha([t_0/2,t_0]\times\R^6)}
		\lesssim \|\vv^m f \|_{C^\alpha([t_0/2,t_0]\times\R^6)}.
\]
This inequality follows from the easy-to-establish fact that $|z-z'|\lesssim \rho(z,z')$ when $t=t'$ and $|z-z'|\leq 1$.  Putting this together with \Cref{l:weighted-Schauder} and using the fact that $q(\alpha,\gamma,k,m) \leq - (2+\gamma)_+$, we have
\[
	\| \vv^{m + (2+\gamma)_+} D^2_v f \|_{L^\infty}
		\leq \| \vv^{m - q(\alpha,\gamma,k,m)} D^2_v f \|_{L^\infty}
		\lesssim \frac{1}{t_0^{1-\frac{\alpha^2}{6-\alpha}}} (1 + \|\vv^m f\|_{C^\alpha([t_0/2,t_0]\times \R^6)})^{p(\alpha)}.
\]
Thus,
\[
	J_2
		\lesssim \frac{1}{t_0^{1- \frac{\alpha^2}{6-\alpha}}}
			\left(\|\vv^m f\|_{L^\infty_t C^\alpha_{x,v}([t_0/2,t_0]\times\R^6)}^{p(\alpha)} +1\right)
			\|g(t_0,\cdot)\|_{L^\infty(\R^6\times B_1(0)^2)}^{1/2},
\]
Then, using \Cref{lem:holder_char} and the fact that $\|g(t_0,\cdot)\|_{L^\infty(\R^6\times B_1(0)^2)} = \|g\|_{L^\infty([0,t_0]\times \R^6\times B_1(0)^2)}$, since $\overline G$ is increasing, we find 
\[
	J_2
		\lesssim \frac{1}{t_0^{1- \frac{\alpha^2}{6-\alpha}}} (1 + g)^{\frac{p(\alpha)+1}{2}},
\]
which concludes the bound of $J_2$.

\medskip

\noindent {\it Step 5: Establishing~\eqref{e:g_comparison} and concluding the proof.} 
Putting together the bounds on $J_1, \dots, J_5$ with~\eqref{e:holder_c3}, we find, at $(t_0,x_0,v_0,\chi_0,\nu_0)$,
\begin{equation}\label{e:dtg}
	\partial_t g
		\lesssim g + \frac{1}{t_0^{1 - \frac{\alpha^2}{6-\alpha}}}\left(1 + g\right)^{\frac{p(\alpha)+1}{2}}
		\lesssim \frac{1}{t_0^{1 - \frac{\alpha^2}{6-\alpha}}}\left(1 + g\right)^{\frac{p(\alpha)+1}{2}}.
\end{equation}
The second inequality follows by using Young's inequality to show that $g \lesssim 1 + g^{P(\beta)}$.  Choosing $N$ to be larger than the implied constant in \eqref{e:dtg}, we obtain~\eqref{e:g_comparison}.  This yields a contradiction, as outlined above.  Thus, our proof is concluded.
\end{proof}

Finally, combining Lemma \ref{l:D2f-Linfty}, Lemma \ref{l:holder-vx-to-t}, and Proposition \ref{prop:holder_propagation}, we obtain a time-integrable bound on $D_v^2 f$ for our solution:
\begin{lemma}\label{l:time-int}
	Let the assumptions of \Cref{prop:holder_propagation} hold, and let $f$ be the solution constructed in Theorem \ref{t:exist} corresponding to $f_{\rm in}$. Then for $t\in [0,T_H]$, 
\[
	\|\vv^{m+(\gamma+2)_+}D_v^2 f(t,\cdot,\cdot)\|_{L^\infty(\R^6)}
		\lesssim \frac{1}{t^{1 - \frac{\alpha^2}{6-\alpha}}}.
\]
The implied constant depends on $\gamma$, $k$, $m$, $\alpha$, $\|f_{\rm in}\|_{L^{\infty,k}}$, and $\|\vv^m f_{\rm in}\|_{C^\alpha}$.
\end{lemma}

\section{Uniqueness}\label{s:unique}



We are now ready to prove our solutions are unique.  In this section, we allow all implied constants in the $\lesssim$ notation to additionally depend on $\|f\|_{L^{\infty,k}}$ and $\|g\|_{L^{\infty,5+\gamma+\eta}}$, where $k$, $g$, and $\eta$ are given below.

In order to state our result, we clarify the type of weak solution we work with. To use the results of \cite{golse2016,henderson2017smoothing}, we require $g$ to be in the kinetic Sobolev space required in these works; namely, for $\Omega \subset [0,\infty)\times\R^6)$, let
\[
	H^1_{\rm kin}(\Omega) = \{ \phi \in L^2(\Omega) : \nabla_v \phi \in L^2(\Omega), (\partial_t + v\cdot\nabla_x) \phi \in L^2_{t,x}H^{-1}_v(\Omega)\},
\]
and let $H^1_{\rm kin,loc}$ be defined in the obvious way. By a weak solution of \eqref{e:landau}, we mean a solution in the sense of integration against test functions in $H_{\rm kin}^1([0,T_0]\times \R^6)$ with compact support.

First, we show uniformly continuous weak solutions have pointwise regularity:

\begin{lemma}\label{l:gC2}
Suppose that $g \in H^1_{\rm kin,loc}((0,T_0]\times \R^6) \cap L^{\infty,5+\eta}([0,T_0]\times \R^6)$, with $\eta>0$, solves~\eqref{e:landau} weakly and is uniformly continuous on $[0,T_0]\times \R^6$ and $g(0,\cdot,\cdot) = f_{\rm in}$. Then $g \in C^2_{\rm kin, loc}((0,T_0]\times\R^6)$.  
\end{lemma}
\begin{proof}
In view of the arguments used in \Cref{t:exist}, namely an application, up to rescaling, of \cite[Theorem~3]{golse2016} and \cite[Theorem 2.12(a)]{henderson2017smoothing}, it is enough to verify that $M_g^{-1}$, $M_g$, $E_g$, and $H_g$ (recall the notation from \Cref{s:conditional}) are bounded uniformly.  The upper bounds on $M_g$, $E_g$, and $H_g$ follow directly from the $L^{\infty,5+\eta}$ bound on $g$. Therefore, the proof is completed after establishing a uniform positive lower bound on $M_g$ on $[0,T_0]\times\R^6$. 

First, we set some notation.  Let
\[
	A = 2\|\tr(\bar a^g)\|_{L^\infty([0,T_0]\times\R^6)} + 1
		\quad \text{and}\quad
	R = 2 \sqrt{AT_0}.
\]
Next, notice that since $g\in L^{\infty,5+\eta}([0,T_0]\times \R^6)$, $g$ is uniformly continuous, and $M_{f_{\rm in}} > (4\pi/3)\delta r^3$, then there exists $T_1 \in(0,T_0]$ such that, for all $(t,x) \in [0,T_1]\times\R^3$,
\begin{equation}\label{e:c012}
	M_g(t,x)
		\geq \frac{1}{T_1}.
\end{equation}
This concludes the proof on $[0,T_1]\times \R^3$, establishing that $g \in C^2_{\rm kin,loc}((0,T_1]\times\R^6)$.  Now that $g$ is sufficiently regular, we note that a classical comparison principle argument yields $g\geq 0$ on $[0,T_1]\times \R^6$.

We now obtain a lower bound on $M_g$ on $[T_1,T_0]\times\R^3$.  Since we have bounds on $M_g^{-1}$, $M_g$, $E_g$, and $H_g$ on $[0,T_1]$, we may iteratively apply the Harnack inequality for the Landau equation \cite[Theorem~4]{golse2016} in order to find $\e > 0$ such that
\[
	\delta R^2 < g \qquad \text{on } \{T_1\} \times \R^3 \times B_R(0).
\]
Let
\[
	\underline g(t,x,v)
		= \e (R^2 - |v|^2 - A(t-T_1))_+,
\]
and notice that, when $|v|^2 \leq R^2 - A(t-T_1)$,
\begin{equation}\label{e:c013}
	(\partial_t + v\cdot\nabla_x)\underline g
		- \tr(\bar a^g D^2_v \underline g)
		- \bar c^g \underline g
		< 0.
\end{equation}

If $\underline g \leq g$ on $[T_1, T_0]\times\R^6$, it is clear that we are finished since the choice of $R$ and $A$ imply that, for any $(t,x) \in [T_1,T_0]\times\R^3$,
\[
	\int_{B_{R/\sqrt 2}} \underline g(t,x,v) \dd v
		\geq \frac{\e \pi R^5}{6 \sqrt {2}}
		\geq \e R^5/10,
\]
and, hence, $M_g \geq \e R^5/10$ on $[T_1,T_0]\times\R^3$.  This, combined with~\eqref{e:c012} yields the desired lower bound on $M_g$ and would complete the proof.

Hence, we assume that $\underline g \not\leq g$ on $[T_1,T_0]\times\R^6$.  Let
\[
	T_h = \sup\{t\in [T_1,T_0] : \underline g(s,x,v) \leq g(t,x,v) \text{ for all } (s,x,v)\in[T_1,t]\times\R^6\}.
\]

We claim that $T_h = T_0$.  If this were true, our proof would be finished; hence, we suppose it is not true.  At time $T_h$, we have that $g \geq \underline g$, by continuity.  Hence $M_g(T_1,x) \geq \e R^5/10$ for all $x$.  As above, we find $\mu>0$ such that $M_g(t,x)> \mu$ for all $(t,x) \in [T_1, T_1 + \mu]\times \R^3$, and, hence, a $g\in C^2_{\rm kin, loc}$ and a classical comparison principle argument shows that $g\geq 0$ on $[T_1, T_1 + \mu]\times\R^6$.

From above, we have established that $M_g$ is positive on $[0,T_h+\mu]\times \R^6$ and $g\geq 0$ on the same set.  Thus, we may apply the Harnack inequality \cite[Theorem~4]{golse2016} in order to conclude that $g > 0$ on $[0, T_h + \mu]\times \R^6$.

By definition and since $T_h < T_0$, we find $z_0 \in [T_h, T_h+\mu]\times\R^6$ such that $g(z_0) < \underline g(z_0)$.  By the positivity of $g$ and up to recentering, we have that there exists $z_h \in [T_h, T_h + \mu]\times B_R \times\R^3$ such that $g(z_h) = \underline g(z_h)$, while $g \geq \underline g$ on $[0, t_h]\times \R^6$.

Let $\phi = g - \underline g$.  From the work above, we have that $z_h$ is the location of a minimum of $\phi$ where $\phi(z_h) = 0$.  Hence
\[
	(\partial_t + v\cdot \nabla_x)\phi
		- \tr(\bar a^g D^2_v \phi)
		- \bar c^g \phi
		\leq 0
\]
On the other hand, using~\eqref{e:landau} and~\eqref{e:c013}, we find
\[
	(\partial_t + v\cdot \nabla_x)\phi
		- \tr(\bar a^g D^2_v \phi)
		- \bar c^g \phi
		>0,
\]
which contradicts the previous inequality.  We conclude that $T_h = T_0$, and the proof is finished.
\end{proof}

\begin{proposition}\label{prop:uniqueness_small_times}
	Suppose that $\alpha \in (0,1)$, $k$ and $m$ satisfy the conditions in \Cref{prop:holder_propagation} and $m\geq 5$, $0 \leq f_{\rm in} \in L^{\infty,k}(\R^6)$, and $\vv^{m} f_{\rm in} \in C^\alpha(\R^6)$.  Let $f$ be the solution of~\eqref{e:landau} with initial data $f_{\rm in}$ constructed in \Cref{t:exist}, with $T_0>0$ such that $\vv^m f \in C^\alpha([0,T_0]\times\R^6)$.
	
	Let $g$ be a weak solution of \eqref{e:landau} satisfying the hypotheses of Lemma \ref{l:gC2}. 
	Then $f=g$.
\end{proposition}

Before proceeding we comment briefly on the assumptions.  The above is a form of weak-strong uniqueness; that is, the uniqueness holds in a weaker class than the constructed solution.  We believe that this is not the weakest class in which uniqueness holds and, at the expense of more technical arguments, one may remove the added assumption that $g$ is uniformly continuous.  Indeed, if the lower bounds of Lemma \ref{l:lower-bounds} were extended to weak solutions, one would get uniform continuity ``for free'' from the H\"older estimate of~\cite{golse2016}.



\begin{proof}
We assume that $f_{\rm in}\not\equiv 0$.  Indeed if $f_{\rm in} \equiv 0$, uniqueness holds via \cite[Theorem~1.1]{HST2018landau}.  By Lemma \ref{l:gC2}, $g\in C^2_{\rm kin,loc}([0,T_0]\times\R^6)$. 

By Proposition \ref{p:initial} and our assumptions, $f$ and $g$ are both continuous up to $t=0$.  
For a positive function $r\in C(0,T_0] \cap L^1(0,T_0]$ to be determined, let
\[
	w = e^{-\int_0^t r(s) ds}(g-f)
		\quad \text{ and }
	W = \frac 1 2 \langle v\rangle^{10} w^2.
\]
Then a straightforward computation yields that, whenever $W \neq 0$,
\begin{equation}\label{e:uniqueness_difference}
\begin{split}
	 \partial_t \cW + &v \cdot \nabla_x \cW
		= \tr\left( \bar a^g D^2_v \cW\right) - \cW^{-1} \nabla_v  \cW \cdot (\bar a^g \nabla_v \cW)
			+ 10 \vv^{-2} v \cdot (\bar a^g \nabla_v \cW)\\
			& + \left(-35 \vv^{-4} v \cdot  (\bar a^g v) + 5 \vv^{-2} \tr(\bar a^g) + \bar c^g\right) \cW
			+ \vv^{10} w\,\tr\left( \bar a^w D_v^2 f\right)
			+ \vv^{10}w \bar c^w f
			- r \cW.
	\end{split}
\end{equation}

Fix $\epsilon>0$, and assume by contradiction that $\sup_{[0,T_0]\times \R^6} \cW(t,x,v) \geq \eps$.  Up to re-centering the equation as in Step 2 of the proof of Proposition \ref{prop:holder_propagation}, we may assume there exists $z_\e = (t_\e, x_\e, v_\e)\in (0,T_0]\times \R^6$ such that $\cW(z_\e) = \e$ and $\cW(t,x,v)<\eps$ for all $t<t_\eps$ and $(x,v)\in \R^6$.  Note that this reduction uses strongly that the $f$ and $g$ are uniformly continuous for two reasons: (i) the re-centering requires an Arzela-Ascoli-based compactness argument on the $x$-translates of $\cW$ and (2) to conclude that $t_\e > 0$ after the re-centering.

We immediately have that, at $z_\e$, $\nabla_v \cW = 0$ and $D^2_v \cW \leq 0$. Derivatives in $t$ and $x$ may not exist pointwise, but we have $(\partial_t +v\cdot\nabla_x)\cW \geq 0$. (This follows by considering the directional derivative of $W(t,x,v)$ at $z_\eps$ in the direction $(1,v,0)$, since $z_\eps$ is a maximum point in $[0,t_\eps]\times \R^6$.) 
It is also clear that $\cW(z_\e) = \e =  \|\cW\|_{L^\infty([0,t_\e]\times \R^6)}$.  Finally, since $g\in L^{\infty,5+\eta}$, we have that $\bar a^g$ and $\bar c^g$ are bounded according to~\Cref{l:abc}.  Using the above in~\eqref{e:uniqueness_difference} and moving the $r\cW$ term to the left hand side, we find
\begin{equation}\label{e:holder_c13}
	r\cW \lesssim
			 \cW + \vv^{10} |w| |\bar a^w|  |D^2_v f| + \vv^5|w| |\bar c^w|.
\end{equation}

Next, we notice that, at $z_\e$,
\begin{equation}\label{e:holder_c11}
	|\bar a^w|
		\lesssim \langle v\rangle^{7+\gamma} \|w(t_\eps,\cdot,\cdot)\|_{L^\infty(\R^6)}
		= \langle v\rangle^{7+\gamma} w
	\quad \text{ and, similarly,} \quad
	|\bar c^w| \lesssim \vv^5 w.
\end{equation}
Recall that $m \geq 5$, and, thus, $m + (2+\gamma)_+ \geq 7 + \gamma$.  We can also apply \Cref{l:time-int} to obtain 
\begin{equation}\label{e:holder_c12}
	\vve^{7+\gamma} |D^2_v f|
		\leq \|D^2_v f(t_\eps,\cdot,\cdot)\|_{L^{\infty,7+\gamma}(\R^6)}
		\lesssim t_\e^{-1 + \frac{\alpha^2}{6-\alpha}}.
\end{equation}
Using~\eqref{e:holder_c11} and~\eqref{e:holder_c12} in~\eqref{e:holder_c13} and recalling the relationship between $w$ and $W$, we find a constant $C$, such that
\[
	r \cW \leq C \Big(1 + t_\e^{-1 + \frac{\alpha^2}{6-\alpha}}\Big)\cW.
\]
With this value of $C$, we obtain a contradiction by defining
\[
	r(t) = C\Big(2+t^{-1 + \frac{\alpha^2}{6-\alpha}}\Big).
\]
We therefore conclude that $z_\e$ cannot exist.  Since this is true for all $\e$, we find that $\cW = 0$, which implies that $g = f$.
%
\end{proof}

Finally, it is straightforward to prove \Cref{t:unique} from \Cref{prop:uniqueness_small_times}:

\begin{proof}[Proof of \Cref{t:unique}]
By \Cref{prop:holder_propagation}, there exists $T_H >0$ such that $\vv^m f \in \Ckin^\alpha([0,T_H]\times \R^6)$.  Thus, we may apply \Cref{prop:uniqueness_small_times} to conclude the first part of \Cref{t:unique}.

We now consider the case where $f_{\rm in} \in L^{\infty,k'}(\R^6)$ for all $k'$.  Combining \Cref{t:exist}.(iii) and \Cref{t:continuation}, we see that $f$ is smooth and all its derivatives lie in $L^{\infty,k'}([0,T]\times \R^6)$ for all $k'$.  Thus, \Cref{prop:uniqueness_small_times} applies with $T_H = T$, which concludes the proof.
\end{proof}

\appendix

\section{Regularity in $x$ and $v$ implies regularity in $t$}\label{s:reg}

We consider solutions $f$ to the linear equation
\begin{equation}\label{e:landau_div}
\partial_t f + v\cdot \nabla_x f = \Tr(a D^2_v f) + c f,
\end{equation}
where $a$ is a non-negative definite matrix that grows at most like $\langle v\rangle^{2+\gamma}$, $c$ is uniformly bounded, and~\eqref{e:landau_div} enjoys a maximum principle.  All implied constants depend only on the upper bounds of $a$ and $c$, but {\em do not} depend on any lower bound of $a$.

In this appendix, we show the following:
\begin{proposition}\label{prop:holder_in_t}
Suppose that $f \in C^\alpha_{{\rm kin, loc}}([0,T]\times \R^6) \cap L^\infty([0,T]\times \R^6)$ and solves~\eqref{e:landau_div}.  Then, for any $z_0$, we have
\begin{equation*}
	\|f\|_{\Ckin^{\alpha}(Q_1(z_0)\cap [0,T]\times \R^6)}
		\lesssim \langle v_0\rangle^{\alpha (1+\gamma/2)_+} \Big(\|f\|_{L^\infty([0,T]\times \R^6)} +  \sup_{t \in [\max\{0,t_0-1\},t_0]} [f(t,\cdot,\cdot)]_{C_{{\rm kin}, x,v}^{\alpha}(B_2(x_0,v_0))}\Big).
\end{equation*}
The implied constant depends {\em only} on the upper bounds of $a$ and $c$.
\end{proposition}

A useful transformation here is given by
\[
	\mathcal{T}_{z_0}(z) = (t_0 + t, x_0 + x + tv_0, v_0 + v).
\]
Recall also $\delta_r$ from~\eqref{e:rho_transformation}.  The main lemma is the following.
\begin{lemma}\label{lem:scaling_holder}
For any $z_0 \in [0,T]\times \R^6$, $r\in (0,1]$, $t_1 \in[0, \langle v_0 \rangle^{-(2+\gamma)_+}]$ such that $r^2 t_1 + t_0 \in [0,T]$ and $|x_1|, |v_1| < 1$, and for $z_1 = (t_1, x_1, v_1)$, we have
\begin{equation*}
\begin{split}
	|f(\mathcal{T}_{z_0}(\delta_r (z_1))) - f(z_0)|
		\lesssim |r|^\alpha \left(\|f\|_{L^\infty([0,T]\times \R^6)} + [f(t_0,\cdot,\cdot)]_{C^\alpha_{{\rm kin},x,v}(B_2(x_0,v_0))}\right).
\end{split}
\end{equation*}
\end{lemma}

\subsection{Concluding~\Cref{prop:holder_in_t} from~\Cref{lem:scaling_holder}}

\begin{proof}[Proof of \Cref{prop:holder_in_t}]
Fix any $z_1$ and $z_2$ in $Q_1(z_0)$.  Recall that $\rho(z_1,z_2) \approx \rho(z_2,z_1)$ so we may assume that $t_2 \geq t_1$.

We consider first the case when $\rho(z_2, z_1) \geq \langle v_1\rangle^{-(1+\gamma/2)_+}$.  Then we have
\[
	\frac{|f(z_2) - f(z_1)|}{\rho(z_2,z_1)^{\alpha}}
		\leq \langle v_1\rangle^{(1+\gamma/2)_+\alpha} \| f\|_{L^\infty([0,T]\times \R^6)}
		\lesssim \langle v_0\rangle^{(1+\gamma/2)_+\alpha} \| f\|_{L^\infty([0,T]\times \R^6)},
\]
since $\langle v_1\rangle^{(1+\gamma/2)_+} \approx \langle v_0 \rangle^{(1+\gamma/2)_+}$.

Next we consider the case when $\rho(z_2,z_1) < \langle v_1\rangle^{-(1+\gamma/2)_+}$.  Clearly, then, $t_2 - t_1 < \langle v_0\rangle^{-(2+\gamma)_+}$.  Let
\[
	s_2 = \langle v_1 \rangle^{-(2+\gamma)_+},
		\quad r^2 = \frac{t_2-t_1}{s_2},
		\quad w_2 = \frac{v_2-v_1}{r},
			\quad \text{ and } \quad
		y_2 = \frac{x_2 - x_1 - r^2 s_2 v_1}{r^3}.
\]
Notice that $z_2 = \mathcal{T}_{z_1}(s_2,y_2,w_2)$, $|s_2| \leq \langle v_1 \rangle^{-(2+\gamma)_+}$, and $r \in (0,1)$.  Hence, we may apply~\Cref{lem:scaling_holder} to find
\[
	|f(z_2) - f(z_1)|
		\lesssim r^\alpha \left(\|f\|_{L^\infty([0,T]\times \R^6)} + [f(t_1,\cdot,\cdot)]_{C^\alpha_{{\rm kin},x,v}(B_2(x_0,v_0))}\right).
\]
Recalling the definition of $r$, we have
\[\begin{split}
	\frac{|f(z_2) - f(z_1)|}{\rho(z_2,z_1)^\alpha}
		&\lesssim \frac{r^\alpha}{\rho(z_2,z_1)^\alpha} \left(\|f\|_{L^\infty([0,T]\times \R^6)} + [f(t_1,\cdot,\cdot)]_{C^\alpha_{{\rm kin},x,v}(B_2(x_0,v_0))}\right)\\
		&= \frac{|t_2-t_1|^{\alpha/2}}{s_1^{\alpha/2}\rho(z_2,z_1)^\alpha} \left(\|f\|_{L^\infty([0,T]\times \R^6)} + [f(t_1,\cdot,\cdot)]_{C^\alpha_{{\rm kin},x,v}(B_2(x_0,v_0))}\right)
\end{split}\]
The proof is finished after noting that $s_1^{-\alpha/2} \lesssim \langle v_0 \rangle^{\alpha(1 + \gamma/2)_+}$ and $|t_2- t_1|^{\alpha/2} \leq \rho(z_2,z_1)^\alpha$.
\end{proof}

\subsection{The proof of~\Cref{lem:scaling_holder}}

\begin{proof}[Proof of \Cref{lem:scaling_holder}]
The proof is based on a maximum principle argument.  Let $z_0$ be as in the statement of the lemma.  Without loss of generality, we may assume that $t_0 = 0$ and $x_0 = 0$.  Then $z_1 \in [0, \min\{\langle v_0\rangle^{(2+\gamma)_+}, T\}]\times B_1(0) \times B_1(v_0)$.  

\medskip

\noindent {\em Step 1: A cut-off function.} Let $\phi \in C_c^\infty(\R^6)$  be a cut-off function such that
\begin{equation}\label{e:cutoff}
	\begin{split}
		&0 \leq \phi \leq 1,
			\quad \phi \equiv 1 \text{ on } B_1(0),\\
		&\phi\lesssim \langle v\rangle^{-3} \langle x\rangle ^{-3},\\
		&|\partial_{v_i} \phi|\lesssim \langle v\rangle^{-4}\langle x\rangle^{-3} \quad\text{ for all } i = 1, 2, 3,\\
		& |\partial_{x_i} \phi|\lesssim \langle v\rangle^{-3}\langle x\rangle^{-4} \quad\text{ for all } i = 1, 2, 3, \quad \text{ and}\\
		 &|\partial_{v_iv_j} \phi| \lesssim \langle v\rangle^{-5} \langle x\rangle ^{-3} \quad \text{ for all } i,j=1,2,3.
\end{split}
\end{equation}
We also set some useful notation.  For any $r\geq 0$ and any function $g$, let
\[
	g_r(z) = g(\mathcal{T}_{z_0}(\delta_r(z))).
\]

\medskip

\noindent {\em Step 2: An auxiliary function and its equation.} 
Then, let
\[
	\tilde F(z) = f_r(z) - f(0, 0, v_0).
\]
It is straightforward to check that
\[
	\tilde F_t + v\cdot \nabla_x \tilde F - \Tr(a_r D^2_v \tilde F)
		= r^2 c_r f_r.
\]
Let
\[
	\psi(t,x,v) = \phi(x - tv, v)
		\quad \text{ and } \quad
	F = \psi \tilde F.
\]
Again, after a straightforward computation, we find
\begin{equation*}
\begin{split}
	F_t + &v\cdot \nabla_x F
		- \Tr(a_r D^2_v F) + 2 (a_r \nabla_v \log(\psi))\cdot \nabla_v F\\
		&= F\left( 2 \nabla_v \log(\psi)\cdot (a_r \nabla_v \log(\psi)) - \Tr(a_r \psi^{-1}D_v^2 \psi)\right) + r^2\psi c_r f_r.
\end{split}
\end{equation*}

Fix $R = 2\langle v_0 \rangle^3 / r$.  
We claim that there exists a constant $C>0$ such that, on $[0,t_1]\times \R^3 \times B_R(0)$,
\begin{equation}\label{e:holder1}
	2|\nabla_v \log(\psi) \cdot (a_r \nabla_v \log(\psi))| + |\Tr(a_r \psi^{-1}D_v^2 \psi)|
		< C\langle v_0 \rangle^{(2+\gamma)_+}.
\end{equation}
This is established at the end of this proof.

Assuming \eqref{e:holder1}, let
\[
	\overline F(t) = e^{t C \langle v_0\rangle^{(2+\gamma)_+}}\left(\|F(0,\cdot,\cdot)\|_{L^\infty(\R^6)}
		+ \sup_{s\in[0,t_1],x\in\R^3, |v| = R} F(s,x,v)_+
		+ r^2 t \|c\|_{L^\infty}\|f\|_{L^\infty}\right),
\]
Then, since $\overline F \geq 0$, 
\begin{equation*}
\begin{split}
	\overline F_t + &v\cdot \nabla_x \overline F
		- \Tr(a_r D^2_v \overline F) + 2 (a_r \nabla_v \log(\psi))\cdot \nabla_v \overline F\\
		&=  C \langle v_0\rangle^{(2+\gamma)_+} \overline F + r^2 \|c\|_{L^\infty} \|f\|_{L^\infty} e^{C\langle v_0 \rangle^{(2+\gamma)_+} t} \\
		&> \overline F\left( 2 \nabla_v \log(\psi) \cdot( a_r \nabla_v \log(\psi)) - \Tr(a \psi^{-1}D_v^2 \psi)\right) + r^2\psi c_r f_r.
\end{split}
\end{equation*}
Above, we used~\eqref{e:holder1} and that $\|f_r\|_{L^\infty} = \|f\|_{L^\infty}$. Note that $\psi$ is compactly supported in $x$. Thus, by the maximum principle, we find $F(t,x,v) \leq \overline F(t)$ on $[0,t_1]\times \R^3 \times B_R(0)$.  In particular, we have
\begin{equation*}
	F(t_1,x_1,v_1)
		\leq e^{t_1 C \langle v_0\rangle^{(2+\gamma)_+}}\left(\|F(0,\cdot,\cdot)\|_{L^\infty(\R^6)}
		+ \sup_{s\in[0,t_1], x\in\R^3, |v| = R} F(s,x,v)_+
		+ t_1 r^2\|c\|_{L^\infty}\|f\|_{L^\infty}\right).
\end{equation*}
For the lower bound of $F(t_1,x_1,v_1)$, consider
\[\underline F(t,x,v) = -e^{t C\langle v_0\rangle^{(2+\gamma_+}} \left( \|F(0)\|_{L^\infty} + \sup_{s\in[0,t_1], x\in\R^3, |v| = R} F(s,x,v)_-\right).\]
By an argument similar to the one for $\overline F$, but simpler because one can use $0 \leq r^2 \psi c_r f_r$, we have $F(t,x,v) \geq \underline F(t)$ on $[0,t_1]\times \R^3\times B_R(0)$. 
Altogether, we have
\begin{equation}\label{e:holder2}
	|F(t_1,x_1,v_1)|
		\leq e^{t_1 C \langle v_0\rangle^{(2+\gamma)_+}}\left(\|F(0,\cdot,\cdot)\|_{L^\infty(\R^6)}
		+ \sup_{s\in[0,t_1], x\in\R^3, |v| = R} |F(s,x,v)|
		+  r^2\|c\|_{L^\infty}\|f\|_{L^\infty}\right).
\end{equation}

\medskip

\noindent {\em Step 3: Quantitative bounds on the right hand side of~\eqref{e:holder2}.} 
Unpacking the coordinate transformations and using the decay of $\phi$, it is easy to verify that
\begin{equation}\label{e:holder3}
	\|F(0,\cdot,\cdot)\|_{L^\infty(\R^6)}
		\lesssim r^\alpha \left( [f(0,\cdot,\cdot)]_{\Ckin^\alpha(B_2(0,v_0))} + \|f(0,\cdot,\cdot)\|_{L^\infty(\R^6)}\right).
\end{equation}
Indeed, fix any $(x,v) \in \R^6$.  We consider two cases.  First, if $r^3|x|, r|v| \leq 2$, then 
\[\begin{split}
	|\psi(0,x,v)(f_r(0,x,v) &- f(0,0,v_0))|
		\lesssim \langle x\rangle^{-3} \vv^{-3} |f(0, r^3 x, rv + v_0) - f(0,0,v_0)|\\
		&\leq \langle x\rangle^{-3} \vv^{-3}\rho((0,r^3x, rv + v_0), (0,0,v_0))^\alpha [f]_{\Ckin^\alpha(Q_2(0,0,v_0))}\\
		&= \langle x\rangle^{-3} \vv^{-3} r^\alpha(|x|^{1/3} + |v|)^\alpha [f]_{\Ckin^\alpha(Q_2(0,0,v_0))}
		\leq r^\alpha [f]_{\Ckin^\alpha(Q_2(0,0,v_0))}.
\end{split}\]
Second, if $r^3|x|$ or $r|v|$ is greater than $2$, then  
\[
	|\psi(0,x,v)(f_r(0,x,v) - f(0,0,v_0))|
		\lesssim \frac{\|f(0,\cdot,\cdot)\|_{L^\infty(\R^6)}}{\langle x\rangle^3 \langle v\rangle^3}
		\lesssim r^3 \|f(0,\cdot,\cdot)\|_{L^\infty(\R^6)}
		\leq r^\alpha \|f(0,\cdot,\cdot)\|_{L^\infty(\R^6)},
\]
since $r\leq 1$. Hence~\eqref{e:holder3} holds.

Next, we check the boundary term.  Indeed, if $(s,x,v) \in [0,t_1]\times \R^3\times \partial B_R(0)$, then
\[
	|F(s,x,v)|
		= \psi(s,x,v) |\tilde F(s,x,v)|
		\lesssim \frac{\|f\|_{L^\infty}}{\langle x - vs\rangle^3 \langle v\rangle^3}
		\lesssim \frac{r^3\|f\|_{L^\infty}}{\langle v_0\rangle^9}
		\leq r^\alpha \|f\|_{L^\infty}.
\]
We conclude that
\begin{equation}\label{e:holder3b}
	\sup_{s\in[0,t_1], x\in \R^3,|v|=R} |F(s,x,v)|
		\lesssim r^\alpha \|f\|_{L^\infty}.
\end{equation}

Plugging~\eqref{e:holder3} and~\eqref{e:holder3b} into~\eqref{e:holder2}, and using the fact that $t_1 \leq \langle v_0\rangle^{-(2+\gamma)_+}$, we find
\begin{equation}\label{e:holder5}
	|F(t_1,x_1,v_1)|
		\lesssim e^{C \langle v_0 \rangle^{2+\gamma} t} r^\alpha \left( [f(0,\cdot,\cdot)]_{\Ckin^\alpha(B_2(0,v_0))} +  \|f\|_{L^\infty}\right).
\end{equation}

\noindent {\em Step 4: Unpacking the transformations to obtain the desired estimate.}
Since $|t_1|, |x_1|, |v_1| < 1$, then $\psi(t_1,x_1,v_1) = 1$ and $v_1 \in B_R(0)$.  We thus conclude, from~\eqref{e:holder5}, that
\[\begin{split}
	|f_r(t_1,x_1,v_1) - f(0,0,v_0)|
		&= F(t_1,x_1,v_1)
		\lesssim e^{C \langle v_0 \rangle^{2+\gamma} t} r^\alpha \left( [f(0,\cdot,\cdot)]_{\Ckin^\alpha(B_2(0,v_0))} + \|f\|_{L^\infty}\right).
\end{split}\]
Hence, the proof is finished once we establish~\eqref{e:holder1}, which is the last step.

\medskip

\noindent {\em Step 5: The coefficient bounds~\eqref{e:holder1}.} We show that the first term on the left in \eqref{e:holder1} is bounded by $C\langle v_0\rangle^{(2+\gamma)_+}$. We omit the proof for the second term, which is similar

Notice that $\nabla_v\psi(t,x,v) = (\nabla_v - t \nabla_x)\phi( x-vt, v)$.  Using the non-negative definiteness of $a$, we obtain
\begin{equation}\label{e:holder4}
	|\nabla_v \log(\psi)\cdot (a_r \nabla_v \log(\psi))|
		\leq 2 \left( \nabla_v\log(\phi) \cdot (a_r \nabla_v \log(\phi)) + t^2 \nabla_x \log(\phi) \cdot( a_r \nabla_x \log(\phi))\right). 
\end{equation}
In the above, we have abused notation in the following way: $a_r$ is evaluated at $(t,x,v)$ while $\phi$ is evaluated at $(x-vt,v)$.

For the first term on the right in~\eqref{e:holder4}, we find from \eqref{e:cutoff} that
\begin{equation}\label{e:first-term}
	\nabla_v\log(\phi) \cdot (a_r \nabla_v \log(\phi))
		\leq |a_r| |\nabla_v \log(\phi)|^2
		\lesssim \frac{\langle v_0 + rv\rangle^{2+\gamma}}{\langle v\rangle^2}.
\end{equation}
It is straightforward to establish, by considering the four cases when $|rv|$ is comparable or not to $|v_0|$ and when $\gamma \geq -2$ or $\gamma<-2$, that the right hand side is bounded by $\langle v_0\rangle^{2+\gamma}$.

The second term in \eqref{e:holder4} can be handled as follows. When $\gamma < -2$, since all terms are bounded, we find
\[
	t^2 \nabla_x\log(\phi)\cdot (a_r \nabla_x \log(\phi))
		\lesssim 1
		\leq \langle v_0\rangle^{(2+\gamma)_+}.
\]
Next, consider the case when $2 + \gamma \geq 0$. Again using~\eqref{e:cutoff}, we find
\[
	t^2 \nabla_x\log(\phi) \cdot (a_r \nabla_x \log(\phi))
		\leq t^2 |a_r| |\nabla_x \log(\phi)|^2
		\lesssim t^2 \frac{\langle v_0 + rv\rangle^{2+\gamma}}{\langle x-vt\rangle^2}.
\]
This case is more subtle than \eqref{e:first-term} because the $x-vt$ term in the denominator may not be large, even if $v$ is large.  Therefore, we must use the smallness of $t$ to balance the fact that $\langle v_0 + rv\rangle$ may be large.   Recall that $|v| \leq R = 2\langle v_0\rangle^{3}/r$ and $t \leq \langle v_0\rangle^{-(2+\gamma)_+}$.  Since $\langle x-vt\rangle \geq 1$, we have
\[
	t^2 \frac{\langle v_0 + rv\rangle^{2+\gamma}}{\langle x-vt\rangle^2}
		\lesssim \langle v_0\rangle^{-2(2+\gamma)} \langle v_0\rangle^{3(2+\gamma)}
		= \langle v_0\rangle^{2+\gamma}
		= \langle v_0\rangle^{(2+\gamma)_+},
\]
as desired.  This completes the proof of \eqref{e:holder1}, which completes the proof of the lemma.
\end{proof}

\section{Interpolation Lemmas}\label{s:interp}

Here, we collect the technical interpolation lemmas used in Section \ref{s:holder}. The first lemma is for functions on $\R^3$ with the standard Euclidean H\"older metric:

\begin{lemma}\label{lem:holder_interpolation} 
Suppose that $d\in \N$ and $\phi:\R^d \rightarrow \R$ is a bounded $C_{\rm loc}^{2,\beta}$ function, with $\beta \in (0,1)$.  For any $\alpha \in (0,1)$, 
\begin{equation}\label{e:D2_Holder_Holder_interpolation}
	\|D^2 \phi \|_{L^\infty(B_1(z))}
		\lesssim [\phi]_{C^\alpha(B_2(z))}
			+ [\phi]_{C^\alpha(B_2(z))}^\frac{\beta}{2 + \beta - \alpha}
				[D^2 \phi]_{C^\beta(B_2(z))}^{ 1- \frac{\beta}{2+\beta - \alpha}}.
\end{equation}
\end{lemma}

\begin{proof}
 Let $z\in \R^d$ be a given point. Taking $x \in B_1(z)$ and $h\in \R^d$ sufficiently small,
we see that
\begin{equation*}
 \phi(x+h)-\phi(x) = \nabla \phi(\xi_1) h = \nabla \phi(x) h + (\nabla \phi(\xi_1)- \nabla \phi(x)) h,
\end{equation*}
where $\xi_1 = x + \theta_1 h$ for some $\theta_1 \in (0,1)$. Let $\sigma>0$ be a constant, to be chosen later. If $|h| \leq \sigma$, then
\begin{equation*}
 |\nabla \phi(\xi_1)-\nabla \phi(x)| \leq \sigma \| D^2 \phi \|_{L^\infty(B_{\sigma}(x))}.
\end{equation*}
If we let $h = \sigma \nabla \phi(x)/|\nabla \phi(x)|$, we then get
\begin{equation}\label{e:D1_bound_interpolation}
	|\nabla \phi(x)|
		\leq \sigma^{\alpha-1} [\phi]_{C^\alpha(B_{\sigma}(x))}
			+ \sigma \| D^2 \phi \|_{L^\infty(B_{\sigma}(x))}.
\end{equation}
Now, taking the Taylor expansion of $\phi$ to first order about a point $x \in B_1(z)$ gives
\begin{equation*}
	 \phi(x+h) = \phi(x) + \nabla \phi(x) h + \frac 1 2 h D^2 \phi (\xi_2) h,
\end{equation*}
where $\xi_2 = x + \theta_2 h$ for some $\theta_2 \in (0,1)$. We rewrite this as
\begin{equation*}
	h D^2 \phi(x) h
		= 2(\phi(x+h) - \phi(x)) - 2\nabla \phi(x) h + h (D^2 \phi(x) - D^2 \phi(\xi_2)) h.
\end{equation*}
Since $D^2\phi(x)$ is a real symmetric matrix, its matrix norm is given by its largest
eigenvalue, which is achieved by some unit eigenvector $\hat v$. With $r>0$ to be chosen later, setting $h = r \hat v$ yields
\begin{equation*}
\begin{split}
	r^2 | D^2 \phi(x) |
		&= 2 (\phi(x+h) - \phi(x)) - 2\nabla \phi(x) h + h (D^2 \phi(x) - D^2 V(\xi)) h \\
		&\leq 2r^\alpha [\phi]_{C^\alpha(B_r(x))} + 2 r \| \nabla \phi \|_{L^\infty(B_r(x))}
			+ r^{2+\beta} [D^2 \phi]_{C^\beta(B_r(x))}.
\end{split}
\end{equation*}
Taking the supremum in $x$ over $B_1(z)$, we get that
\begin{equation*}
	\| D^2 \phi \|_{L^\infty(B_1(z))}
		\lesssim r^{\alpha-2} [\phi]_{C^\alpha(B_{1+r}(z))}
			+r^{-1} \| \nabla \phi \|_{L^\infty(B_{1+r}(z))}
			+ r^\beta [D^2 \phi]_{C^\beta(B_{1+r}(z))}.
\end{equation*}
Using \eqref{e:D1_bound_interpolation} on the middle term yields
\begin{equation}\label{e:mixed_interpolation_bound}
\begin{split}
	\| D^2 \phi \|_{L^\infty(B_1(z))}
		&\lesssim \left( r^{\alpha-2} + r^{-1}{\sigma}^{\alpha-1}\right)[\phi]_{C^\alpha(B_{1+r+\sigma}(z))}\\
		&\quad\quad\quad+\sigma r^{-1} \| D^2 \phi \|_{L^\infty(B_{1+r+\sigma}(x))}
+ r^\beta [D^2 \phi]_{C^\beta(B_{1+r}(z))}.
\end{split}
\end{equation}
We can then 
take $\sigma = c r$ for $c<1$ small enough that we can absorb the middle term into the norm on the left hand side.  As long as $r < 1/2$, this yields
\begin{equation*}
	\|D^2 \phi \|_{L^\infty}
		\lesssim r^{\alpha-2}  [\phi ]_{C^\alpha(B_2(z))}
			+ r^\beta  [D^2 \phi]_{C^\beta(B_2(z))}.
\end{equation*}
Thus, choosing
\begin{equation*}
	r
		= \min \left\{
			1/2,
			\ \  \left( [\phi ]_{C^\alpha(B_2(z))} /
				[D^2 \phi]_{C^\beta(B_2(z))}\right)^\frac{1}{2 + \beta - \alpha} \right\},
\end{equation*}
then we obtain \eqref{e:D2_Holder_Holder_interpolation}.
\end{proof}

%

Our next, somewhat elementary lemma allows us to trade regularity (in the form of a larger H\"older exponent) for pointwise decay.  It is stated for the (time-independent) kinetic H\"older spaces, however, it is straightforward to prove it for the time-dependent kinetic and standard H\"older spaces as well.  

\begin{lemma}\label{lem:holder_decay}
	Suppose that $\phi:\R^6\to \R$ is such that $\phi\in L^{\infty,k_1}(\R^6)$ and $\vv^{k_2}\phi\in \Ckin^\alpha(\R^6)$, for some $\alpha\in (0,1)$ and $k_1 \geq k_2 \geq 0$.  If $\beta \in (0,\alpha)$ and $\ell \in [k_2, k_1]$ are such that
	\[
		\ell
			\leq k_1 \left(1 - \frac{\beta}{\alpha}\right) + k_2 \frac{\beta}{\alpha},
	\]
	then
	\[
		\|\langle v\rangle^\ell \phi\|_{\Ckin^{\beta}(\R^6)}
			\lesssim [\vv^{k_2}\phi]_{\Ckin^\alpha(\R^6)}^\frac{\beta}{\alpha} \|\phi\|_{L^{\infty,k_1}(\R^6)}^{1-\frac{\beta}{\alpha}}.
	\]
	The same result is true with the standard H\"older spaces replacing the kinetic ones.
\end{lemma}

\begin{proof} 


We omit the proof with the standard H\"older spaces since it is exactly the same.  Let
\[
	R = \left(\vv^{k_2 - k_1} \|\phi\|_{L^{\infty,k_1}(\R^6)} [\vv^{k_2}\phi]_{\Ckin^\alpha(\R^6)}^{-1}\right)^{1/\alpha}.
\]
Let $(x_1,v_1), (x_2,v_2) \in \R^6$. If $|x_1-x_2|^{1/3}+|v_1-v_2| \geq R$, then
\[\begin{split}
	\frac{|\phi(x_1, v_1) - \phi(x_2, v_2)|}{(|x_1-x_2|^{1/3}+|v_1-v_2|)^\beta}
		&\lesssim \langle v\rangle^{- k_1} \frac{2 \|\phi\|_{L^{\infty, k_1}(\R^6)}}{R^\beta}
		= \frac{[\vv^{k_2} \phi]_{\Ckin^\alpha(B)}^\frac{\beta}{\alpha} \|\phi\|_{L^{\infty,k_1}(\R^6)}^{1-\frac{\beta}{\alpha}}}{\langle v\rangle^{k_1\left(1 - \frac{\beta}{\alpha}\right) + k_2 \frac{\beta}{\alpha}}}.
\end{split}\]

On the other hand, suppose that $|x_1-x_2|^{1/3}+|v_1-v_2| < R$.  Let $\theta = |x_1-x_2|^{1/3}+|v_1-v_2|$.  We have 
\[
	\frac{|\phi(x_1, v_1) - \phi(x_2, v_2)|}{\theta^\beta}
		= \frac{|\phi(x_1, v_1) - \phi( x_2, v_2)|}{\theta^\alpha} \frac{\theta^\alpha}{\theta^\beta} \lesssim \langle v\rangle^{-k_2} [\vv^{k_2} \phi]_{\Ckin^\alpha(B)} R^{\alpha - \beta}.
\]
From the definition of $R$, we obtain
\[
	\frac{|\phi(x_1, v_1) - \phi(x_2, v_2)|}{\theta^\beta}
		\leq \frac{[\vv^{k_2}\phi]_{\Ckin^\alpha(B)}^\frac{\beta}{\alpha} \|\phi\|_{L^{\infty,k_1}(\R^6)}^{1-\frac{\beta}{\alpha}}}{\langle v\rangle^{k_1\left(1 - \frac{\beta}{\alpha}\right) + k_2\frac{\beta}{\alpha}}}.
\]
\end{proof}

\bibliographystyle{abbrv}
\bibliography{landau}

\end{document}